\documentclass[10pt,reqno,final]{amsart}
\usepackage{amsmath,amsthm}
\usepackage{amsfonts,amssymb,mathrsfs}
\usepackage[notcite,notref]{showkeys}
\usepackage{fancybox,subfigure,version,pifont,booktabs}
\usepackage{url}
\usepackage{multirow}
\usepackage{graphicx}
\usepackage{ifpdf,color}

\numberwithin{equation}{section}
 \numberwithin{table}{section}
 \numberwithin{figure}{section}

\newtheorem{thm}{Theorem}[section]

\newtheorem{lem}[thm]{Lemma}
\newtheorem{prop}[thm]{Proposition}

\theoremstyle{remark}
\newtheorem{rem}{Remark}[section]

\makeatletter
\newcommand{\bddots}{%
  \mathinner{\mkern1mu\raise\p@\vbox{\kern7\p@\hbox{.}}\mkern2mu
    \raise4\p@\hbox{.}\mkern2mu\raise7\p@\hbox{.}\mkern1mu}}
\makeatother

\def\II{\operatorname{I\!I}}

\newcommand{\Lebsp}[1]{L^{#1}}
\newcommand{\Lspace}[2][1]{L^{#1}\!\left(#2\right)}
\newcommand{\Lspacewt}[3][1]{\Lebsp{#1}_{#3}\!\left(#2\right)}

\newcommand{\SspaceQwt}[2]%[\omega]
[\chi]{\widetilde{\mathrm{H}}^{#2}_{#1}\!\left(\mcq\right)}

\newcommand{\Polsp}[1]{\ensuremath{\mathcal{P}_{#1}}}
\newcommand{\Pspace}[2][N]{\ensuremath{\Polsp{#1}(#2)}}%{\ensuremath{\Polsp{#1}\!\left(#2\right)}}

\newcommand{\diff}[1]{\ensuremath{\mathrm{d}#1}}
\newcommand{\gradient}[1]{\ensuremath{\nabla#1}}

\newcommand{\p}[3][]{\ensuremath{\frac{\partial^{#1}#2}{\partial #3^{#1}}}}
\newcommand{\psub}[3][]{\ensuremath{\partial^{#1}_{#3}#2}}
\newcommand{\inprod}[3][]{\ensuremath{\left\langle#2, #3\right\rangle_{#1}}}
\renewcommand{\iint}{\ensuremath{\int\hspace*{-8pt}\int}}
\newcommand{\backmatter}{\setcounter{section}{0}\renewcommand{\thesection}{\Alph{section}}}

\renewcommand{\qed}  {\hfill $\rule{2.3mm}{2.3mm}$}

\def \mct {{\bigtriangleup}}
\def \mcq {{\Box}}
\renewcommand{\inprod}[3][]{\ensuremath{\left(#2, #3\right)_{#1}}}

\allowdisplaybreaks

 \ifpdf
 \DeclareGraphicsExtensions{.pdf,.png,.jpg}
 \else
 \DeclareGraphicsExtensions{.eps}
 \fi
 \graphicspath{{latexfile/figfiles/}}

\date{\today}
\author[M.~Samson, H.~Li \& L.~Wang  ]{Michael Daniel Samson${}^1$, \hspace*{10px}  Huiyuan Li${}^2$\hspace*{10px}and\hspace*{10px}Li-Lian Wang${}^1$ }
\title[ A New TSEM I: Implementation $\&$ Analysis]{A New Triangular Spectral Element Method  I: Implementation and Analysis on a Triangle}
\thanks{%\footnotetext{
${}^1$ Division of Mathematical Sciences, School of Physical and Mathematical Sciences,  Nanyang Technological University, 637371, Singapore. The research of the authors is partially supported by Singapore AcRF Tier 1 Grant RG58/08.
\\
\indent ${}^2$ Institute of Software, Chinese Academy of Sciences, Beijing 100190, China. The work of  this
author was supported by National Natural Science Foundation of
China (NSFC) Grants 10601056 and 10971212.
}
 \keywords{Rectangle-triangle mapping, consistency condition, triangular spectral elements, spectral accuracy} \subjclass{65N35, 65N22,65F05, 35J05}

\begin{document}
\graphicspath{{./figfiles/}}

\baselineskip 14pt
\begin{abstract} This paper serves as our first effort to develop a new triangular spectral element method (TSEM) on unstructured meshes, using the rectangle-triangle mapping proposed in the conference note  \cite{LWLM}.
Here,  we provide some new insights into the originality and distinctive features of the mapping, and show that this transform only induces a logarithmic singularity, which allows us to  devise a fast, stable and accurate numerical algorithm for its removal. Consequently, any triangular element can be treated as efficiently as a quadrilateral element, which affords a great flexibility in handling complex computational domains. Benefited from the fact that the image of the mapping includes the polynomial space as a subset, we are able to obtain optimal $L^2$- and $H^1$-estimates of approximation by the proposed  basis functions  on triangle. The implementation details and
some numerical examples are provided to validate the efficiency and accuracy  of the proposed method. All these will pave the way for  developing  an unstructured TSEM based on, e.g.,  the hybridizable discontinuous Galerkin formulation.
\end{abstract}
\maketitle
%\tableofcontents

\vspace*{-10pt}
\section{Introduction}
\setcounter{equation}{0}

The spectral element method (SEM), originated from Patera \cite{patera1984spectral}, integrates the unparalleled accuracy of spectral methods with the geometric flexibility of finite elements, and also enjoys  a high-level parallel computer architecture.  Nowadays, it has become a pervasive numerical technique for simulating challenging problems in complex domains \cite{DFM02,CHQZ06b}.  While the classical SEM on quadrilateral/hexahedral elements (QSEM) exhibits the advantages of using tensorial basis functions and naturally diagonal mass matrices, the need for high-order methods on unstructured
meshes with robust adaptivity spawns the development of triangular/tetrahedral spectral elements.  In general, research efforts along this line  fall into three trends: (i) nodal TSEM based on high-order polynomial interpolation on special interpolation points \cite{CB,Hesthaven,TWV}; (ii) modal TSEM based on the Koornwinder-Dubiner polynomials \cite{Koornwinder,Dubiner,Kar.S05}; and (iii) approximation by non-polynomial functions \cite{SWL,3DLW10,SXL11}.

The question of how to construct ``good'' interpolation points for stable high-order polynomial interpolation on the triangle is still quite subtle and somehow open.  The strict analogy of the Gauss-Lobatto integration rule on quadrilaterals/hexahedra does not exist on triangles \cite{Helen09}, though a ``relaxed'' rule can be constructed in the sense of \cite{XuYuan2011}. We refer to \cite{PasRapp10} for an up-to-date review and a very dedicated comparative study of various criteria for constructing workable interpolation points on the triangle.  In general, such points have low degree of precision (i.e., exactness for integration of polynomials), and this motivates \cite{Pas.R06} the use of a different set of points for integration, which are mapped from the Gauss points on the reference square via the Duffy's transform \cite{Duffy.82}.  The development of TSEM using  Koornwinder-Dubiner polynomials as modal basis functions, generated by the rectangle-triangle mapping (i.e., the Duffy's transform), can be best attested to by the monograph \cite{Kar.S05} and spectral-element package {\em NekTar} (http://www.nektar.info/). The  analysis of this approach can be found in e.g., \cite{Guo.W05,Schwab,LS10,Ch11}.
 However, the main drawbacks of this approach lie in that the interpolation points are unfavorably clustered near one vertex of the triangle, and there is no corresponding nodal basis, making it complicate to implement.  To overcome the second difficulty, a full tensorial rational approximation on triangles was proposed in \cite{SWL} for elliptic problems, and extended to the Navier-Stokes problem in \cite{SXL11}.  This approach still builds on the collapsed Duffy's transform with clustered grids.

It is important to point out that the Duffy's transform not only leads to undesirable distributions of  interpolation/quadrature points, but also requires modifying the tensorial polynomial basis to meet the underlying consistency  conditions (analogous to ``pole conditions'' in polar/spherical coordinates) induced by the singularity of the transform.  Our mind-set is therefore driven by searching for a method based on a different rectangle-triangle mapping that can lead to favorable  distributions  of interpolation/quadrature points on the triangle without loss of accuracy and efficiency of implementation.  With this in mind, we introduced in the conference note \cite{LWLM} a new mapping that pulls one side (at the middle point) of the triangle to two sides of the rectangle (cf. Figure \ref{mapLGL} (a)), and results in  much more desirable distributions of the mapped LGL points (cf. Figure \ref{mapLGL} (c) vs. (d)). Moreover, this mapping is one-to-one.

 \begin{figure}[!ht]
\subfigure[]{
\begin{minipage}[t]{0.24\textwidth}
\centering
\rotatebox[origin=cc]{-0}{ \includegraphics[width=1\textwidth]{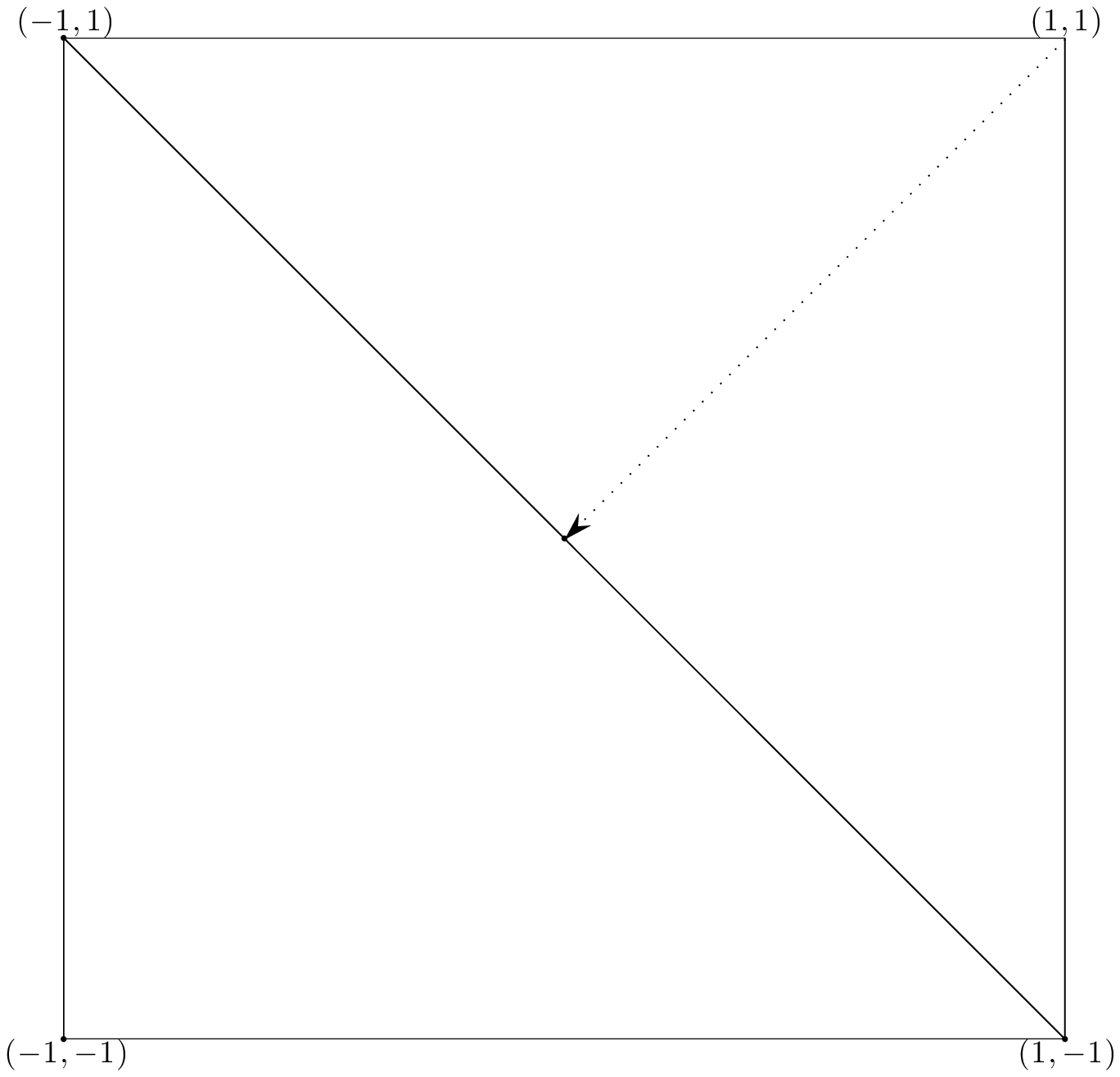}}
\end{minipage}}
\subfigure[]{
\begin{minipage}[t]{0.24\textwidth}
\centering
\rotatebox[origin=cc]{-0}{\includegraphics[width=1\textwidth]{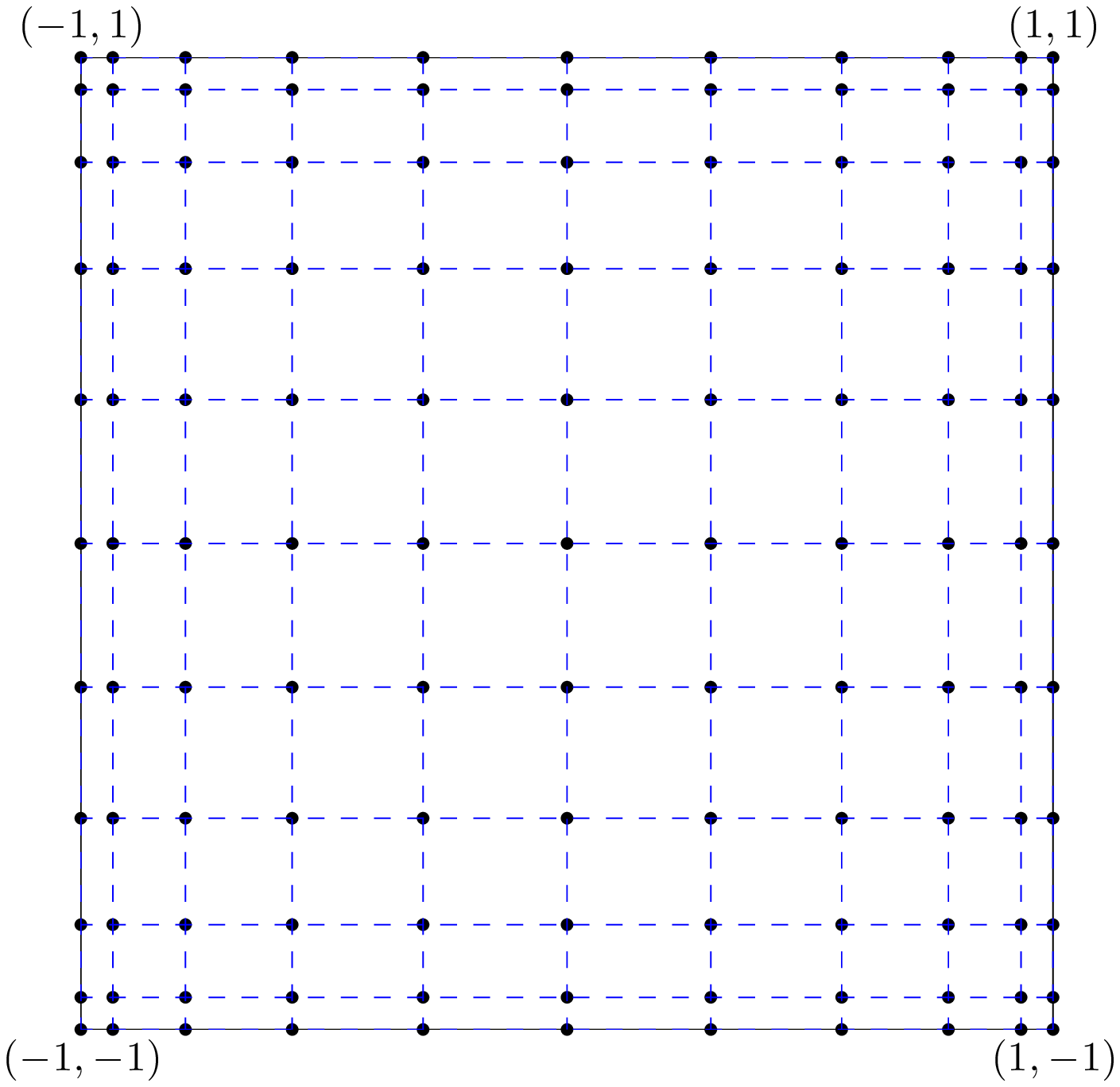}}
\end{minipage}}
\subfigure[]{
\begin{minipage}[t]{0.23\textwidth}
\centering
\rotatebox[origin=cc]{-0}{ \includegraphics[width=1\textwidth]{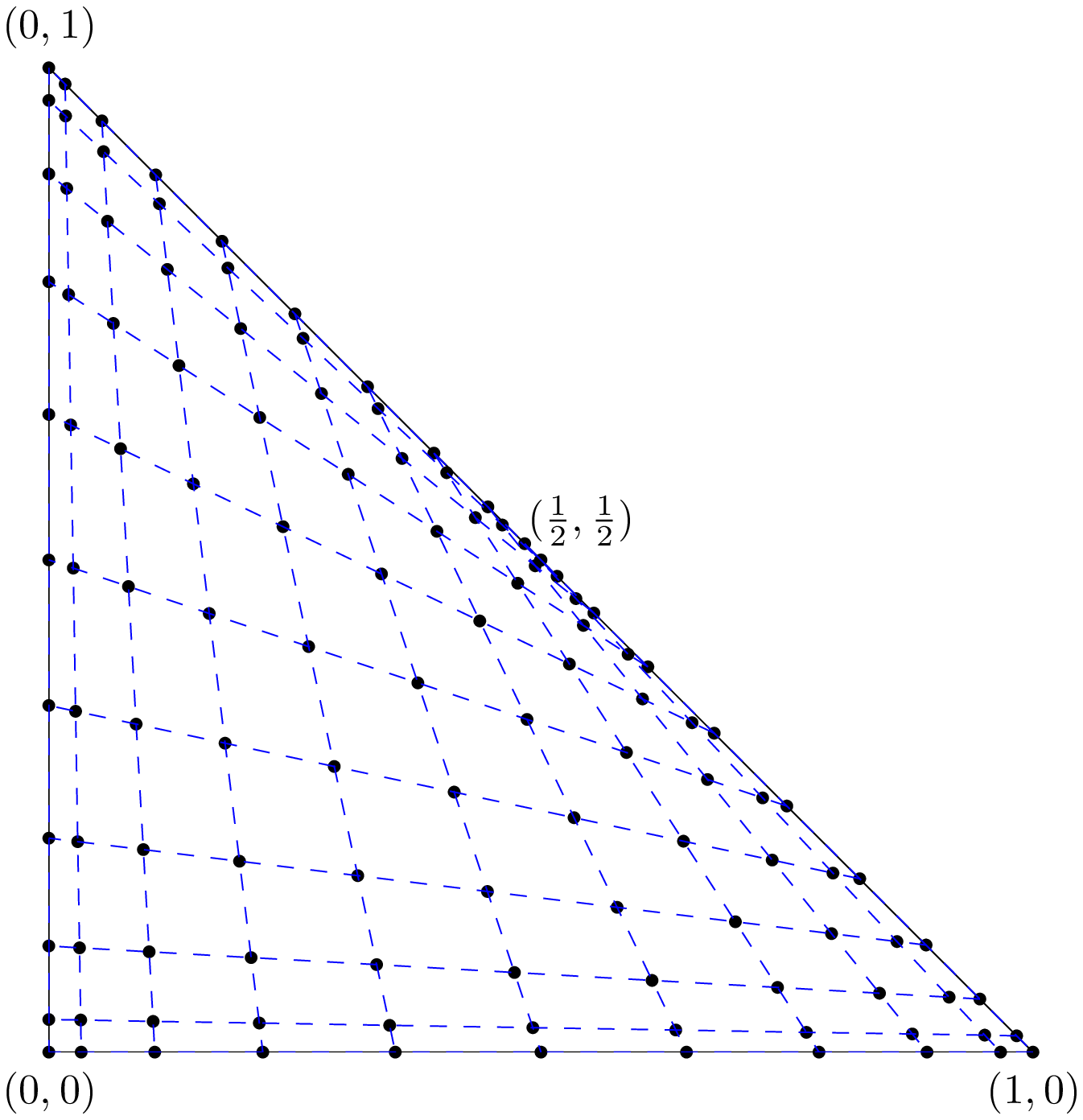}}
\end{minipage}}
\subfigure[]{
\begin{minipage}[t]{0.23\textwidth}
\centering
\rotatebox[origin=cc]{-0}{\includegraphics[width=1\textwidth]{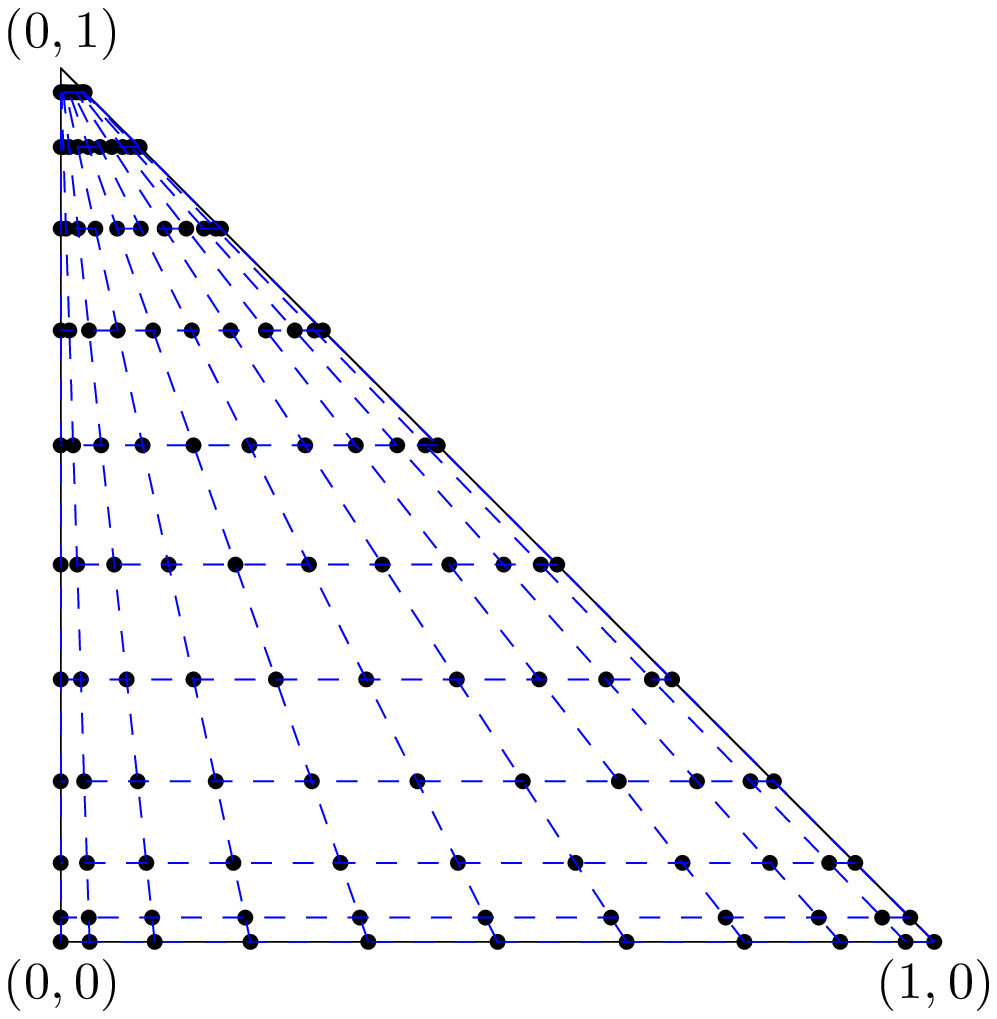}}
\end{minipage}}
\caption{\small(a). $\mct \leftrightarrow \mcq$ mapping; (b). tensorial Legendre-Gauss-Lobatto (LGL) points on $\mcq$; (c). mapped LGL grids on $\mct$; (d). mapped LGL grids on $\mct$ using the Duffy's transform.}\label{mapLGL}
\end{figure}

The purposes of this paper are threefold: (i) have  some new insights into this mapping; (ii)  demonstrate that the singularity of the mapping is of logarithmic type, which can be fully removed; and (iii)  derive optimal error estimates for approximation by the  associated basis functions.   This work will pave the way for  developing a new TSEM on unstructured meshes, which will be explored in the second part. It also brings about an important viewpoint that any triangular element can be mapped to the reference square via a composite of the rectangle-triangle  mapping and an affine mapping, and with the successful removal of the singularity, the triangular element can be  treated as efficiently as a quadrilateral element. One implication is that this allows a mixture of triangular and quadrilateral elements, so one can handle more complex domains with more regular computational meshes, e.g., by tiling the triangular elements along the boundary of the domain. More importantly,  for general unstructured triangular meshes, we can formulate the underlying variational problems using the recently enhanced hybridizable discontinuous Galerkin methods \cite{HDGUnified, CGorHDG, HDGM}. We expect that the QSEM will enjoy a minimal communication between elements, and  a minimum number of  globally coupled degree of freedoms, and  allow for implementing a large degree of nonconformity across elements (e.g., the hanging nodes and mortaring techniques).  We leave this development to the second part after this work.

The rest of this paper is organized as follows.  In Section 2, we present some new insights of the rectangle-triangle mapping.  In Section 3, we introduce the basis functions and the efficient algorithm for computing the stiffness and mass matrices with an emphasis on how to remove the singularity of the rectangle-triangle transform.  We derive  some optimal approximation results in Section 4,  followed by numerical results on a triangle in Section 5.

\section{The rectangle-triangle mapping}

We collect in this section some  properties of  the rectangle-triangle mapping introduced in
\cite{LWLM}, and provide some insightful perspectives on this transform.

\subsection{The rectangle-triangle mapping} Throughout the paper, we denote by
$$
\mct := \big\{(x, y) :\, 0 < x, y, x + y < 1\big\},\quad \mcq := \big\{(\xi, \eta):\, -1 < \xi, \eta < 1\big\},
$$
 the \emph{reference triangle} and  the \emph{reference square}, respectively.  The rectangle-triangle transform (cf. \cite{LWLM}) $T: \mcq\to \mct,$  takes the form
 \begin{equation}\label{best_map}
	x =\frac18 (1 + \xi)(3 - \eta),\quad y =\frac18 (3 - \xi)(1 + \eta), \quad \forall \; (\xi, \eta) \in {\mcq},
\end{equation}
with the inversion $T^{-1}: \mct\to \mcq:$
\begin{equation}\label{best_map_inv}
	\begin{cases}
		 \xi = 1 + (x - y) - \sqrt{(x - y)^2 + 4(1 - x - y)},\\[0.4em]
		\eta = 1 - (x - y) - \sqrt{(x - y)^2 + 4(1 - x - y)},
	\end{cases}
\end{equation}
for any $(x, y) \in {\mct}.$
It maps the vertices
 $(-1,-1), (1,-1)$ and $(-1,1) $  of the square $\mcq$ to
 the vertices $(0,0),\, (1,0) $ and $(0,1)$ of the triangle $\mct$,  respectively,
while the middle point $(1/2,1/2)$  of the hypotenuse  is the image of the vertex $(1,1)$ of $\mcq$.
In other words, this mapping deforms two edges ($\xi=1$ and $\eta=1$) of $\mcq$
into the  hypotenuse  of $\mct$, see Figure \ref{mapLGL} for illustration.

 Under this mapping,  we have
\begin{equation}\label{xy_xieta_inv}
\p{x}{\xi} = \frac{3-\eta}{8}, \quad \p{x}{\eta} = -\frac{1 + \xi}{8},
	\quad \p{y}{\xi} = -\frac{1 + \eta}{8}, \quad \p{y}{\eta} = \frac{3-\xi}{8},
\end{equation}
and the  Jacobian is  given by
\begin{equation}\label{Jacobiannew}
	J = \det\left(\p{(x, y)}{(\xi, \eta)}\right) = \frac{2 - \xi - \eta}{16} = \frac{\sqrt{(x - y)^2 + 4(1 - x - y)}}{8}:=\frac{\chi}{8}.
\end{equation}

For convenience of presentation, we use the handy notation:
\begin{equation}\label{hndnt}
\widetilde \nabla=(\partial_\xi,\partial_\eta),\quad  \widetilde \nabla^\bot=(-\partial_\eta, \partial_\xi), \quad
\widetilde {\nabla}^\intercal = (1 - \xi)\partial_\xi - (1 - \eta)\partial_\eta,
\end{equation}
where we put `` $\widetilde {~}$ "   to distinguish them from  the differential operators in $(x,y).$
Given  $u(x,y)$ on $\mct,$ we define the transformed function: $\tilde{u}(\xi, \eta) = (u\circ T)(\xi, \eta)= u(x, y),$ and likewise for $\tilde v,$ etc..  Then  we have
\begin{equation}\label{massmatrix}
	({u},{v})_{\mct} = \iint_\mct u(x,y) v(x,y) \diff{x}\diff{y} = \iint_\mcq \tilde{u}(\xi,\eta)\tilde{v}(\xi,\eta)J\, \diff{\xi}\diff{\eta}.
\end{equation}
Moreover,  one verifies that
\begin{equation}\label{grad_u_v}
%\begin{split}
	\gradient{u} %&
	= \left(\psub{u}{x}, \psub{u}{y}\right) =
	\chi^{-1}\big(2(\widetilde{\nabla}\cdot\tilde{u}) + (\widetilde{\nabla}^\intercal\tilde{u}),\, 2(\widetilde{\nabla}\cdot\tilde{u}) - (\widetilde{\nabla}^\intercal\tilde{u})\big),
\end{equation}
and
\begin{equation}\label{grad_u_grad v}
\begin{split}
	\inprod[\mct]{\gradient{u}}{\gradient{v}} &= \iint_\mcq\big(\widetilde{\nabla}\cdot\tilde{u}\big)\big(\widetilde{\nabla}\cdot\tilde{v}\big)\chi^{-1}\diff{\xi}\diff{\eta} + \frac{1}{4}\iint_\mcq\big({\widetilde \nabla}^\intercal\tilde{u}\big)\big(\widetilde {\nabla}^\intercal\tilde{v}\big)\chi^{-1}\diff{\xi}\diff{\eta}.
\end{split}
\end{equation}

We observe from \eqref{grad_u_v}-\eqref{grad_u_grad v} that if $\gradient{u}$ is continuous at the middle point $\left({1}/{2}, {1}/{2}\right)$ of the hypotenuse of $\mct$, there automatically holds (note:
$({\widetilde \nabla}^\intercal\tilde{u})|_{(1,1)}=0$):
\begin{equation}\label{ptsingular}
\big(\widetilde\nabla \cdot \tilde u\big)|_{(1,1)}=\left.\left(\p{\tilde{u}}{\xi} + \p{\tilde{u}}{\eta}\right)\right|_{(1, 1)} =0,
\end{equation}
which is referred to as the consistency  condition,  and can be viewed as an analogy of the pole condition in the polar/spherical coordinates. In general, we have to build the condition \eqref{ptsingular} in the approximation space so as to obtain high-order accuracy, which therefore results in the reduction of dimension and modification of the usual basis functions  (cf. \cite{LWLM}).

One important goal of this paper is to demonstrate that this singularity can be removed, thanks to  the
 observation:
\begin{equation}\label{observe}
\iint_{\mcq}\frac 1 {2-\xi-\eta}\, \diff{\xi}\diff{\eta}=4\ln 2,
\end{equation}
which  implies that for any $f\in C(\overline \mcq),$
\begin{equation}\label{observe1}
\Big|\iint_{\mcq}\frac {f(\xi,\eta)} {2-\xi-\eta}\, \diff{\xi}\diff{\eta}\Big|\le 4M \ln 2,
\end{equation}
where $M=\max_{\overline \mcq}|f(\xi,\eta)|.$ In particular,  the coordinate singularity can be eliminated, if
 $f$ is a polynomial on $\mcq$ (see Subsection \ref{mathcomput}).

%\subsection{Some important properties}

Now, we present other important features of this mapping.  Hereafter,
 let $I=(-1,1),$  and for any integer $N\ge 1,$ let ${P}_N(I)$ be the set of all algebraic polynomials  of degree at most $N$.  Denote by
\begin{equation}\label{mmpn}
{\mathcal P}_N(\mct):={\rm span}\big\{x^iy^j : 0\le i+j\le N\big\}, \quad {\mathcal Q}_N(\mcq):=(P_N(I))^2.
\end{equation}
The following property shows the correspondence between two polynomial spaces.
\begin{prop}\label{tranfm} We have
	\begin{itemize}
		\item[(i)\;]  ${\mathcal P}_N(\mct)\circ T\subset  {\mathcal Q}_N(\mcq).$
				\item[(ii)\;]  ${\mathcal Q}_N(\mcq)= \big({\mathcal P}_N(\mct)\circ T\big) \oplus \chi  \big({\mathcal P}_{N-1}(\mct)\circ T\big).$
\end{itemize}
Here, $T$ is the rectangle-triangle transform defined by \eqref{best_map}, and $\chi=(2-\xi-\eta)/2.$
\end{prop}
\begin{proof} We find from \eqref{best_map} that for $0\le i+j\le N,$
\begin{align*}
      x^i y^j =\Big(\frac{1 + \xi} 2\Big)^i \Big(\frac{3 - \eta} 4\Big)^i \Big(\frac{3 - \xi} 4\Big)^j \Big(\frac{1 + \eta} 2\Big)^j\in {\mathcal Q}_N(\mcq).
\end{align*}
This leads to the inclusion in (i).

We see that  for $0\le i+j\le N-1,$
\begin{align*}
      x^i y^j\chi  =\Big(\frac{1 + \xi} 2\Big)^i \Big(\frac{3 - \eta} 4\Big)^i \Big(\frac{3 - \xi} 4\Big)^j \Big(\frac{1 + \eta} 2\Big)^j \frac{2-\xi-\eta} 2
      \in {\mathcal Q}_N(\mcq),
\end{align*}
which implies  $\chi  \big({\mathcal P}_{N-1}(\mct)\circ T\big)\subset {\mathcal Q}_N(\mcq).$

It remains to prove  ${\mathcal Q}_N(\mcq)\subset \big({\mathcal P}_N(\mct)\circ T\big) \oplus \chi  \big({\mathcal P}_{N-1}(\mct)\circ T\big),$ which we will  show by induction. Firstly, by \eqref{best_map_inv}, it is true for $\xi,\eta,$ so is $\xi\eta,$  since $\xi\eta=5 - 4x - 4y - 2\chi.$ Now, assume that it holds for  $\xi^i\eta^j$ with $0 \leq i, j \le N-1$.  Then,  for $0 \leq i, j \le  N$, we find that  $\xi^N\eta^j = \xi(\xi^{N- 1}\eta^j)$, $\xi^i\eta^N = \eta(\xi^i\eta^{N-1})$, and $\xi^N\eta^N = (\xi\eta)(\xi^{N- 1}\eta^{N - 1})$ are all of the form $(a + bx + cy + d\chi)(p(x, y) + q(x, y)\chi)$, where $a, b, c, d$ are constants,  $p \in
\Pspace[N-1]{\mct}$ and  $q \in \Pspace[N-2]{\mct}$. It is apparent that
	\begin{align*}
		 &(a + bx  + cy + d\chi)(p + q\chi)=(a + bx + cy)p + dp\chi + (a + bx + cy)q\chi + dq\chi^2\\
	&\quad \overset{(\ref{best_map_inv})}	= (a + bx + cy)p + d\big((x - y)^2 + 4(1 - x - y)\big)q+\big(dp + (a + bx + cy)q\big)\chi.
	\end{align*}
	Since $(a + bx + cy)p, d\chi^2q \in \Pspace{\mct}$ and $dp, (a + bx + cy)q \in \Pspace[N- 1]{\mct}$, we have
 $$\xi^N\eta^j, \xi^i\eta^N, \xi^N\eta^N \in  \big({\mathcal P}_N(\mct)\circ T\big) \oplus \chi  \big({\mathcal P}_{N-1}(\mct)\circ T\big),$$
 for all $0 \leq i, j \le N$.
 This completes the induction.
\end{proof}
%\begin{rem}\label{DuffyProp} A similar analysis for Duffy's transform. ????
%\end{rem}

In what follows, let  $\omega>0$ be a generic weight function on $\Omega=\mct$ or $\mcq.$
The weighted Sobolev space $H^r_\omega(\Omega)$ with $r\ge 0$ is defined as in Adams \cite{Adam75}, and its norm and semi-norm are  denoted by $\|\cdot\|_{r,\omega,\Omega}$  and $|\cdot|_{r,\omega,\Omega},$ respectively.   In particular,  if $r=0,$ we denote the inner product and norm of  $L^2_\omega (\Omega) $ by $(\cdot,\cdot)_{\omega, \Omega}$ and $\|\cdot\|_{\omega, \Omega},$ respectively.  Moreover, if $\omega\equiv 1,$ we drop it from the notation.

\begin{prop}\label{equivnormlem} For any $u\in H^1(\mct),$ we have
	\begin{equation}\label{EquivNorm2}
\frac{\sqrt 3} 2 \big\|\widetilde \nabla \cdot \tilde u\big\|_{\chi^{-1},\mcq}+\frac {\sqrt 2} {4}
\big\|\widetilde \nabla^{\bot}\cdot\tilde u  \big\|_{\chi,\mcq} \le \big\|\nabla u\big\|_{\mct}\le  \frac{\sqrt 5} 2 \big\|\widetilde \nabla \cdot \tilde u\big\|_{\chi^{-1},\mcq}+\frac 1 2
\big\|\widetilde \nabla^{\bot}\cdot\tilde u  \big\|_{\chi,\mcq},
	\end{equation}
where $\chi=(2-\xi-\eta)/2,$ $\tilde u=u\circ T$ and the differential operators are defined in \eqref{hndnt}.
\end{prop}
\begin{proof} By  \eqref{grad_u_grad v}%-\eqref{G1G2}
, we have
$$ \big\|\nabla u\big\|_{\mct}^2= \big\|\widetilde \nabla \cdot \tilde u\big\|^2_{\chi^{-1},\mcq}+\frac 1 {4}
 %\big\|(1-\xi)\partial_\xi \tilde u-(1-\eta) \partial_\eta \tilde u\big\|^2_{\chi^{-1},\mcq}.
  \big\|\widetilde{\nabla}^\intercal\tilde{u}\big\|^2_{\chi^{-1},\mcq}.
%  \norm[\chi^{-1}, \mcq]{\widehat{\nabla}^\intercal\cdot\tilde{u}}^2.
$$
Then using the identity:
\begin{align*}
\widetilde {\nabla}^\intercal\tilde{u}
%\widehat{\nabla}^\intercal\cdot\tilde{u}
= (1-\xi)\partial_\xi \tilde u-(1-\eta) \partial_\eta \tilde u= \frac 1 2 \big(2\chi (\widetilde \nabla^{\bot}\cdot\tilde u)-(\xi-\eta)(\widetilde \nabla \cdot \tilde u)\big),
\end{align*}
we obtain
\begin{equation}\label{EquivNorm}
\big\|\nabla u\big\|_{\mct}^2= \big\|\widetilde \nabla \cdot \tilde u\big\|^2_{\chi^{-1},\mcq}+\frac 1 {16}
 \big\|2\chi (\widetilde \nabla^{\bot}\cdot\tilde u)-(\xi-\eta)(\widetilde \nabla \cdot \tilde u)   \big\|^2_{\chi^{-1},\mcq}.
	\end{equation}
As  $|\xi-\eta|\le 2,$ we get
\begin{align*}
\big\|2\chi (\widetilde \nabla^{\bot}\cdot\tilde u)-(\xi-\eta)(\widetilde \nabla \cdot \tilde u)   \big\|^2_{\chi^{-1},\mcq}\le  4  \big\|\widetilde \nabla^{\bot}\cdot\tilde u  \big\|_{\chi,\mcq}^2+  4   \big\|\widetilde \nabla \cdot \tilde u\big\|_{\chi^{-1},\mcq}^2.
\end{align*}
Thus, the upper bound of \eqref{EquivNorm2} is a consequence of  \eqref{EquivNorm}.

It is clear that
$$
-4(\xi-\eta)\chi  (\widetilde \nabla^{\bot}\cdot\tilde u) (\widetilde \nabla\cdot\tilde u) \ge -\big( 2{\chi^2}   |\widetilde \nabla^{\bot}\cdot\tilde u|^2+ {2(\xi-\eta)^2}  |\widetilde \nabla\cdot\tilde u|^2 \big).
$$
Thus,
\begin{align*}
\big(2\chi (\widetilde \nabla^{\bot}\cdot\tilde u)-(\xi-\eta)(\widetilde \nabla \cdot \tilde u)\big)^2& \ge
2{\chi^2}   |\widetilde \nabla^{\bot}\cdot\tilde u|^2-(\xi-\eta)^2  |\widetilde \nabla\cdot\tilde u|^2\\
& \ge 2{\chi^2}   |\widetilde \nabla^{\bot}\cdot\tilde u|^2-4|\widetilde \nabla\cdot\tilde u|^2,
\end{align*}
which implies
$$
 \big\|2\chi (\widetilde \nabla^{\bot}\cdot\tilde u)-(\xi-\eta)(\widetilde \nabla \cdot \tilde u)   \big\|^2_{\chi^{-1},\mcq}\ge
 2 \big\|\widetilde \nabla^{\bot}\cdot\tilde u  \big\|_{\chi,\mcq}^2- 4 \big\|\widetilde \nabla\cdot\tilde u  \big\|_{\chi^{-1},\mcq}^2.
$$
Therefore, the lower  bound of \eqref{EquivNorm2} follows from \eqref{EquivNorm}.
\end{proof}

\begin{rem}\label{spsmap}   We find from  Proposition \ref{equivnormlem}  that under the rectangle-triangle
 mapping  \eqref{best_map},  the space $H^1(\mct)$ is mapped to the weighted space on $\mcq$:
\begin{equation}\label{sp_sob_h1}
	\widetilde H^1_\chi(\mcq):= \big\{\tilde{u} \in \Lspacewt[2]{\mcq}{\chi} : \widetilde \nabla\cdot {\tilde{u}} \in \Lspacewt[2]{\mcq}{\chi},  \widetilde \nabla ^\bot\cdot {\tilde{u}} \in \Lspacewt[2]{\mcq}{\chi^{-1}}\big\},
\end{equation}
and vice verse. \qed
\end{rem}

\subsection{Some new perspectives and a comparison study}  Next, we have some insights of the rectangle-triangle mapping  and compare it with the Duffy's transform \cite{Duffy.82}.% from various perspectives.

Firstly, the transform \eqref{best_map} is a special case of the general mapping
$T_\theta : \mcq \mapsto \mct:$
\begin{equation}\label{genfundmap}
(x,y) = \left(
		\frac{1 + \xi}{2}   \frac{2- (1 - \theta)(1 + \eta)}{2} ,
		\frac{1 + \eta}{2}  \frac{2- \theta(1 + \xi)}{2} \right),
	\quad \forall (\xi, \eta) \in {\mcq},
\end{equation}
with $\theta=1/2.$  We see that  this mapping pulls the hypotenuse of $\mct$ into two edges of $\mct$
at the point $(\theta, 1-\theta).$   The limiting case with $\theta=0$ reduces to  the Duffy's transform:
$T_D : \mcq \mapsto \mct:$
  \begin{equation}\label{duffy_map}
		x =\frac14 (1 + \xi)(1 - \eta) ,\quad
		y = \frac12 (1 + \eta), \quad \forall\, (\xi,\eta)\in \mcq,
\end{equation}
with the inverse transform: $T_D^{-1} : \mct \mapsto \mcq:$
  \begin{equation*}\label{duffy_map_inv}
		\xi =\frac {2x}{1-y} -1 ,\quad \eta = 2y-1, \quad \forall \, (x, y)\in \mct.
\end{equation*}
It collapses one edge, $\eta = 1$, of  $\mcq$ into the vertex
$(0, 1)$ of  $\mct.$ As the singular vertex corresponds to one edge,  the  Duffy's transform is not a one-to-one mapping,
as opposite to \eqref{best_map}.   This results in  a large portion of mapped LGL points clustered near the singular vertex of $\mct$ (see Figure \ref{mapLGL} (d)).   The Jacobian of \eqref{duffy_map}-\eqref{duffy_map_inv} is
  $J=(1-\eta)/8,$ and we have
      \begin{equation}\label{nablau}
      \nabla u=\Big(\frac 4 {1-\eta} \partial_\xi \tilde u, \frac{2(1+\xi)}{1-\eta}\partial_\xi \tilde u+2\partial_\eta \tilde u\Big).
      \end{equation}
Different from \eqref{ptsingular}, the  corresponding consistency  condition of the Duffy's transform   becomes  $ \partial_\xi \tilde u(\xi, 1)=0.$  In a distinct contrast with  \eqref{observe}, the integral
\begin{equation}\label{transs}
\iint_{\mcq} \frac  1 {1-\eta}\, \diff{\xi}\diff{\eta}=\infty.
\end{equation}
Consequently, the consistency  condition has to be built in the approximation space,  and much care has to taken to deal with this singularity for Duffy's transform-based methods in terms of implementation and analysis.

Secondly,  the nature of the point singularity of  \eqref{best_map} is reminiscent to that of the Gordon-Hall mapping  \cite{GordenHall73}, which
maps the reference square  to the unit disc via
\begin{align*}
x = \frac{\xi}{\sqrt{2}} \sqrt{2-\eta^2},\quad
y = \frac{\eta}{\sqrt{2}} \sqrt{2-\xi^2}, \quad \forall\,(\xi,\eta)\in \mcq,
\end{align*}
and whose  Jacobian is  $(2-\xi^2-\eta^2)/\sqrt{(2-\xi^2)(2-\eta^2)}.$ It is clear that this transform induces singularity at four vertices of the reference square (cf. Figure \ref{Gordon-Hall}).  It is worthwhile to point out that  the collocation scheme on the unit disc using this mapping
 was discussed in  \cite{Heinrichs04}, and this mapping technique was further examined in \cite{Boyd11}.

\begin{figure}[th!]
    \hfill%
	\begin{minipage}{0.24\textwidth}
		\includegraphics[width=1\textwidth]{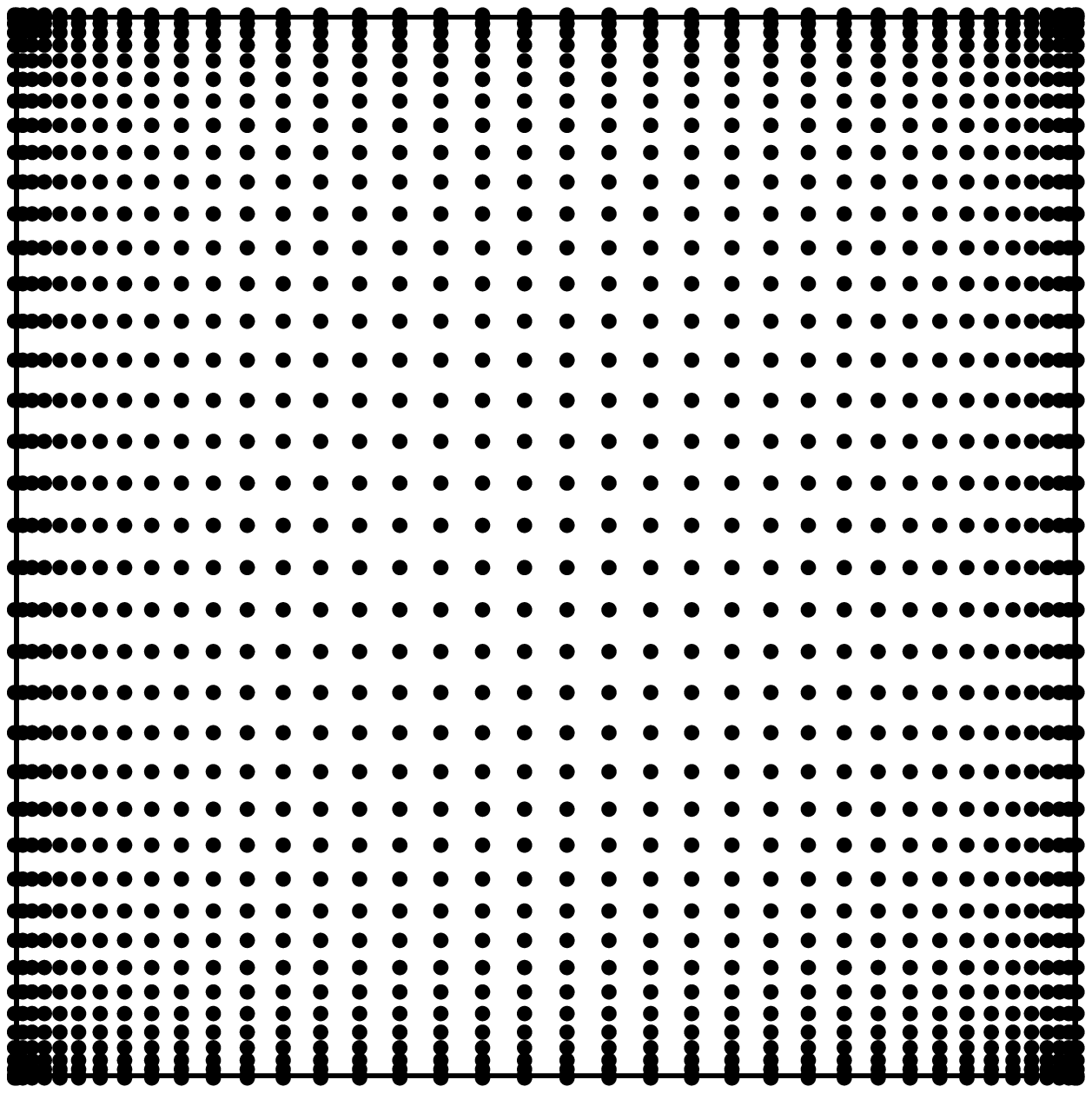}
	\end{minipage}\hfill%
	\begin{minipage}{0.24\textwidth}
		\includegraphics[width=1\textwidth]{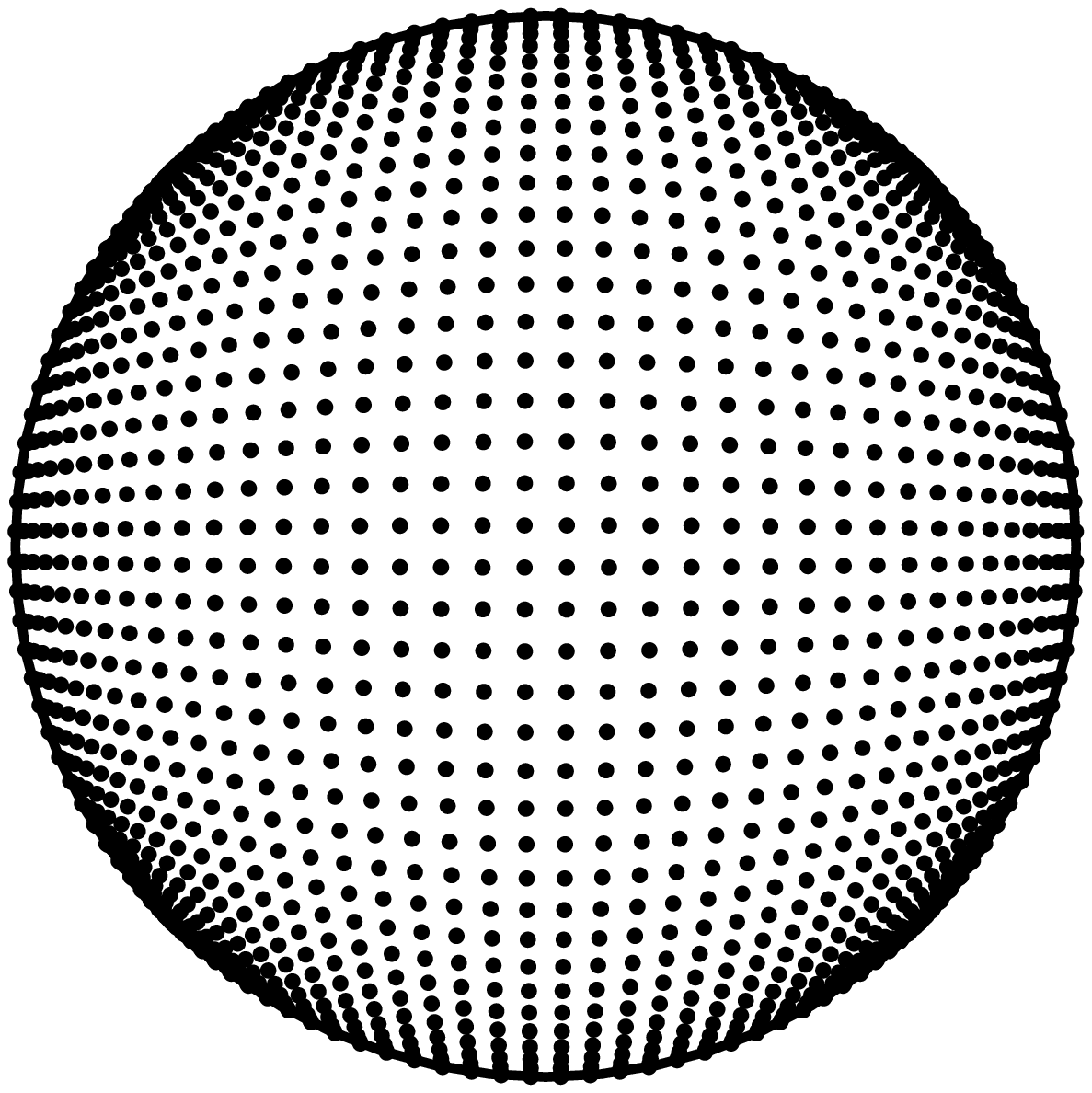}
	\end{minipage}
	\hspace*{\fill}
	\caption{\small Left: tensorial Legendre-Gauss-Lobatto points on the square. Right: the corresponding mapped LGL points on the unit disc.}
	\label{Gordon-Hall}
\end{figure}

In addition, we find  that  the rectangle-triangle transform \eqref{best_map}  can be derived from  the symmetric mapping  on $\mcq:$
\begin{equation}\label{symmap}
	\hat{x} = {\xi} + {\eta}, \quad \hat{y} = {\xi}{\eta}, \quad \forall\, ( \xi, \eta)\in \mcq.
\end{equation}
It transforms any symmetric polynomial in $(\xi, \eta)$ to a polynomial in $(\hat x,\hat y),$ so it is referred to as a symmetric mapping \cite{Weber12}.  One verifies that the image of this mapping is  the curvilinear  triangle
(see Figure  \ref{symmapextfig} (b)):\footnote{It is  worthwhile to note that thanks to the symmetric mapping  $\widehat T : \mcq \mapsto \Omega,$
 Xu \cite{Xu}  discovered the first example of  multivariate Gauss quadrature.}
\begin{align*}
	{\Omega} = \big\{(\hat{x}, \hat{y}) : 1 - \hat{x} + \hat{y}, \, 1 + \hat{x} + \hat{y}, \, \hat{x}^2 - 4\hat{y} > 0 \big\}.
\end{align*}
As the symmetric  mapping \eqref{symmap}, denoted by $\widehat T : \mcq \mapsto \Omega,$ can not distinguish the images of  $(\xi,\eta)$ and $(\eta,\xi),$ it is not one-to-one. To amend this,  one may restrict the domain of $\widehat T$
to the upper triangle, denoted by $\mct_{\rm up},$ in $\mcq$ (see   Figure  \ref{symmapextfig} (a)),  and interestingly,
the square of maximum area  contained in this subdomain is one-to-one mapped to the triangle of maximum area included in the curvilinear triangle $\Omega$, that is,
 \begin{equation}\label{ontot}
 \widehat T :  \widehat \mcq:=(-1,0)\times (0,1) \; \longmapsto\;\; \widehat \mct:=\big\{(\hat{x}, \hat{y}): |\hat{x}| < 1 + \hat{y} < 1\big\},
 \end{equation}
is a bijective mapping (see the shaded parts in Figure  \ref{symmapextfig} (a)-(b)). For clarity of presentation, we denote the
coordinate of any point in  $\widehat \mcq $ by $(\hat \xi,\hat \eta).$    It is clear that the reference square
$\mcq$  and $\widehat \mcq$ are connected by the affine mapping:  $F_1:   \mcq\mapsto \widehat \mcq,$
of the form  (see the shaded parts of Figure  \ref{symmapextfig} (a),  (c)):
\begin{equation}\label{hathat}
\hat \xi=\frac {\xi-1}{2},\quad \hat \eta=\frac{1-\eta} 2,\quad \forall (\xi,\eta)\in \mcq,
\end{equation}
and the affine mapping:   $F_2: \widehat \mct\mapsto \mct,$ takes the form (see the shaded parts of Figure  \ref{symmapextfig} (b), (d)):
\begin{equation}\label{hatmap}
x = \frac{1}{2}(\hat{y}  + \hat{x}+1),\quad  y = \frac{1}{2}(\hat{y} - \hat{x}+1),\quad \forall\, (\hat x,\hat y)\in \widehat \mct.
\end{equation}
In summary,  we  have   $ \mcq \overset{F_1}\longmapsto \widehat \mcq \overset{\widehat T}\longmapsto \widehat  \mct \overset{F_2}\longmapsto  \mct.$ Remarkably, this composite  mapping is identical to the rectangle-triangle mapping
\eqref{best_map}, i.e., $T=F_1\circ \widehat T \circ F_2.$

\begin{figure}[th!]
\begin{minipage}{0.24\textwidth}
	\includegraphics[width=1\textwidth]{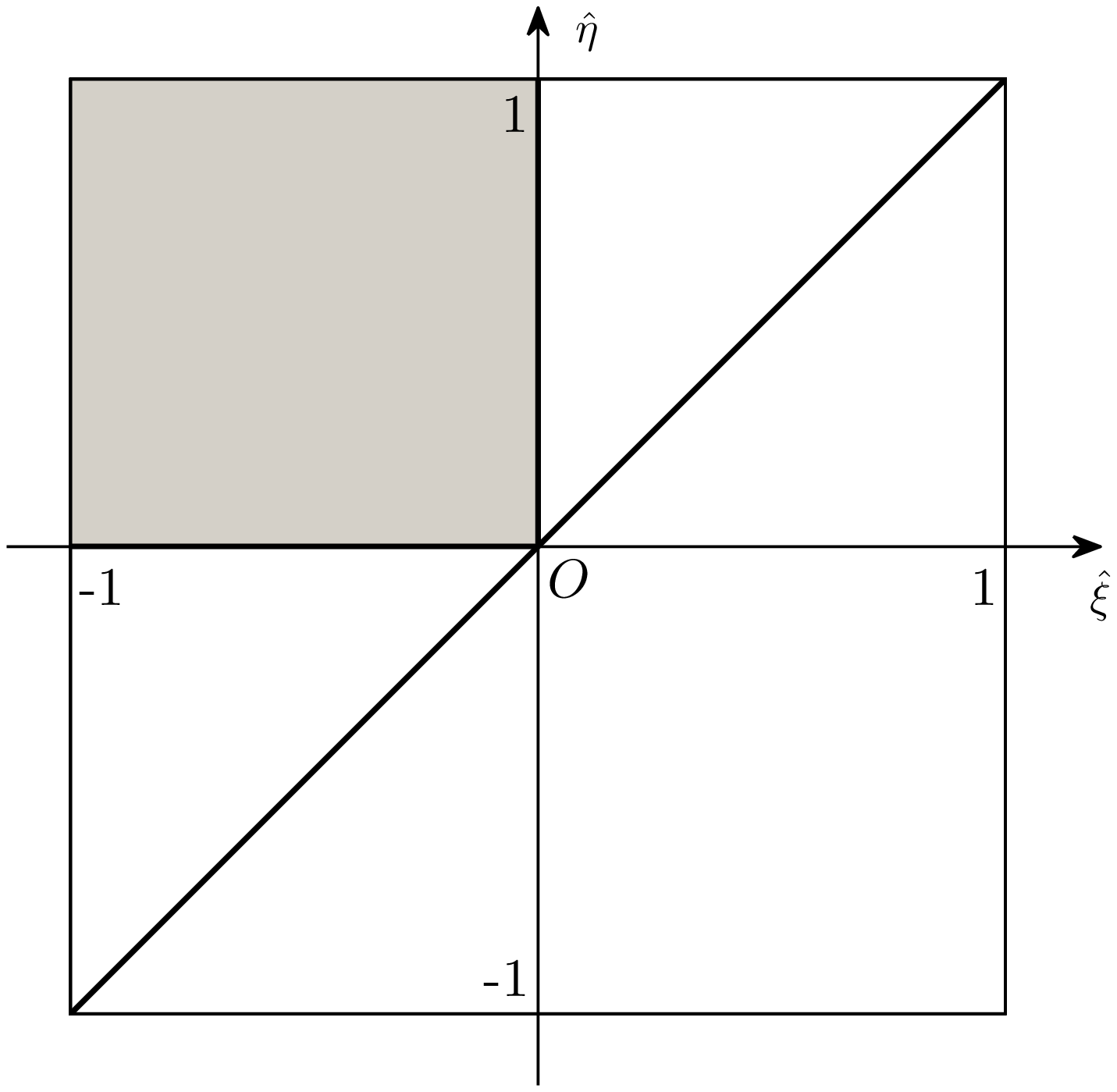}
 \centerline{(a)}
\end{minipage}\ %
\begin{minipage}{0.24\textwidth}
	\includegraphics[width=1\textwidth]{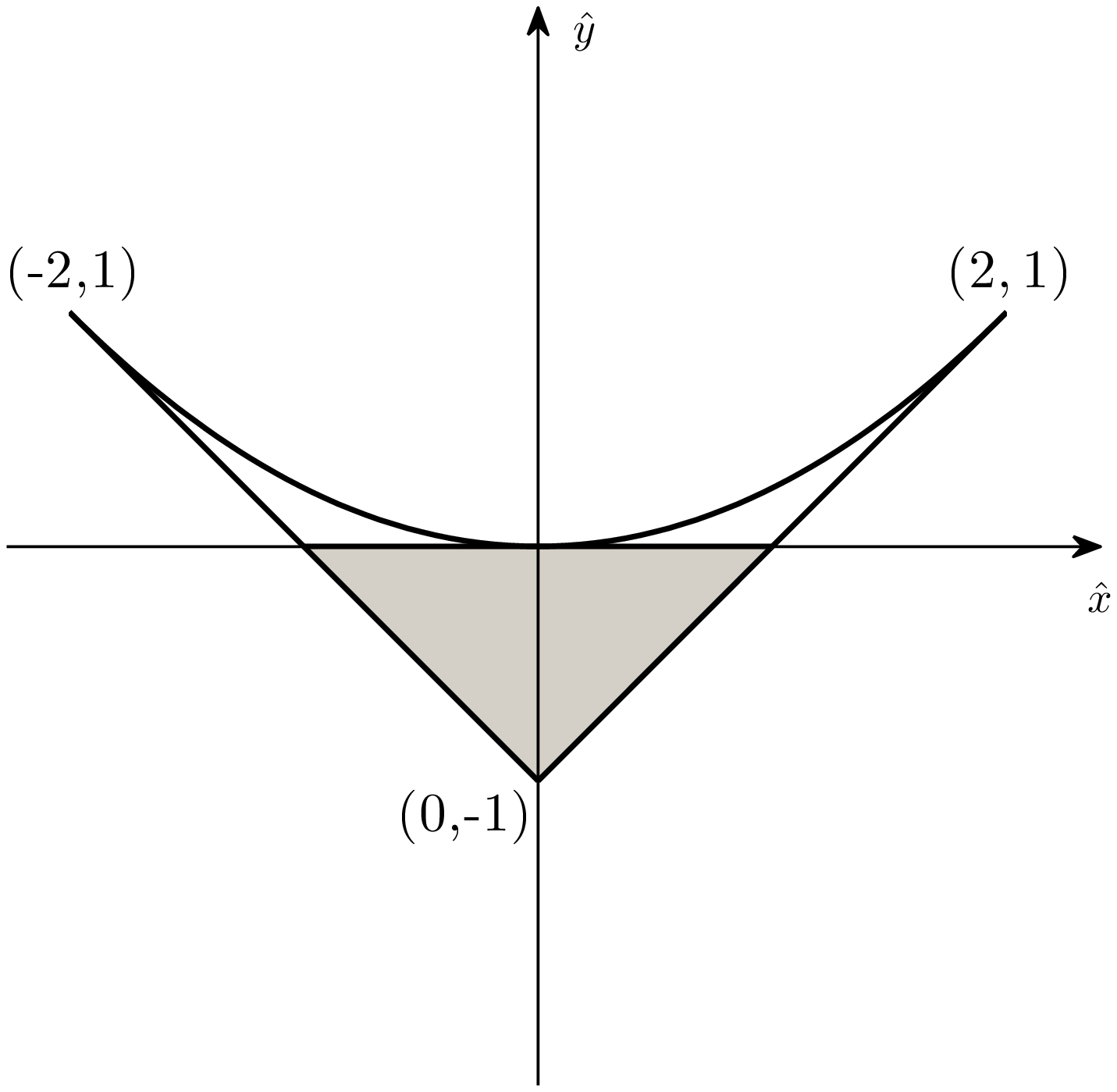}
 \centerline{(b)}
\end{minipage}\ %
\begin{minipage}{0.24\textwidth}
	\includegraphics[width=1\textwidth]{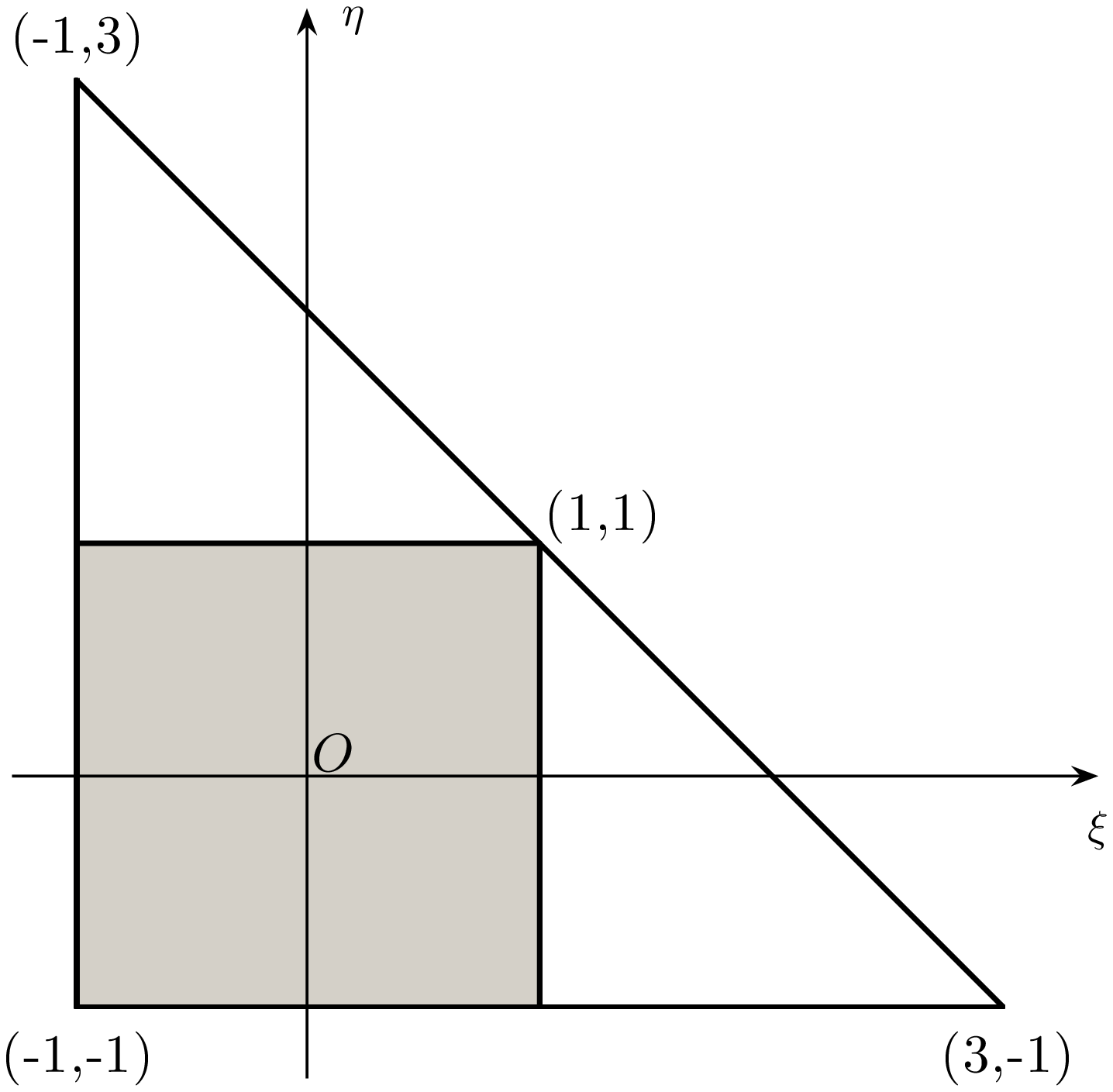}
 \centerline{(c)}
\end{minipage}\ %
\begin{minipage}{0.24\textwidth}
	\includegraphics[width=1\textwidth]{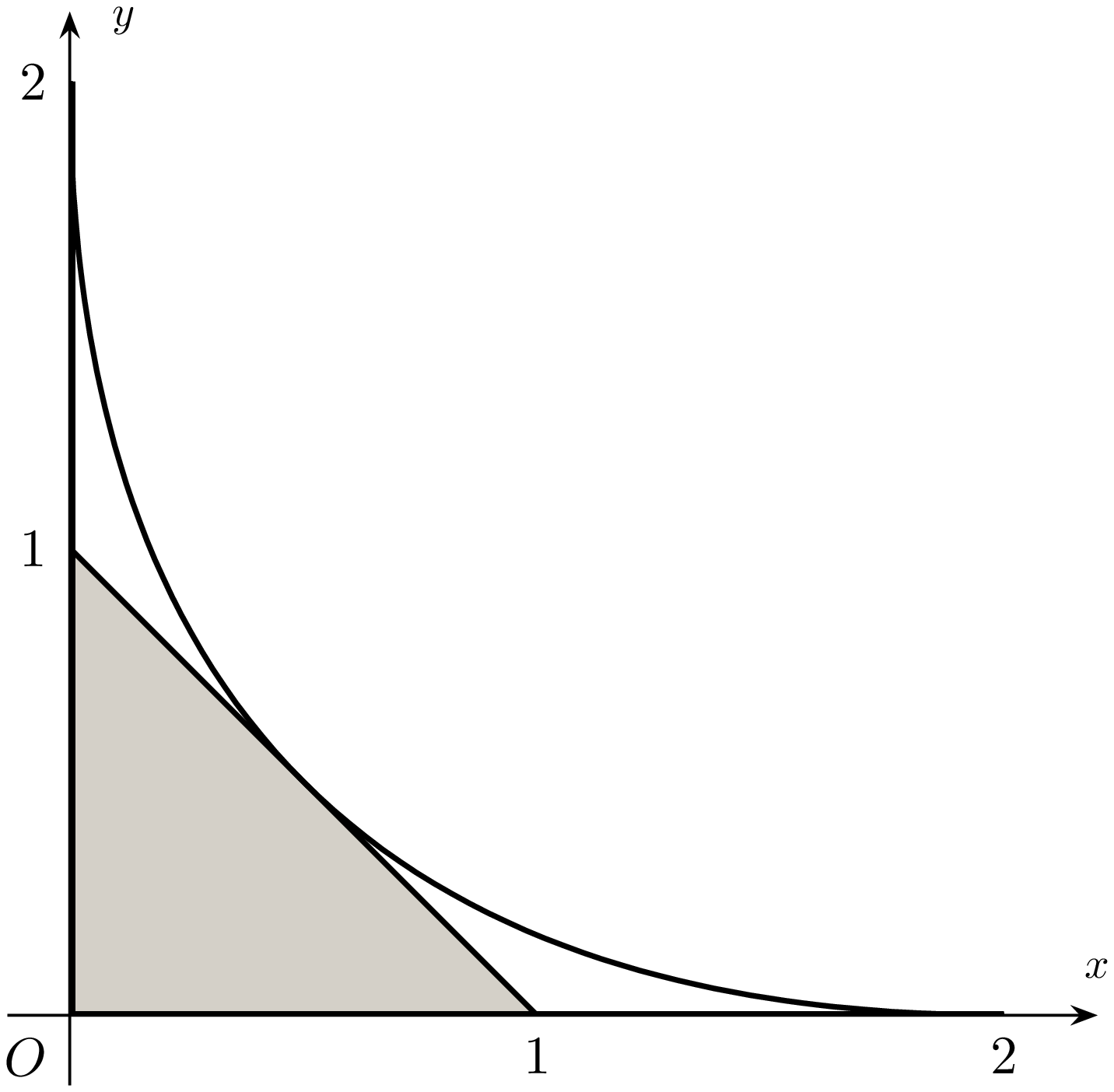}
 \centerline{(d)}
\end{minipage}
\caption{\small (a). The reference square $\mcq,$  the  upper triangle $\mct_{\mathrm{up}} = \left\{(x,y): -1< x< y<1 \right\}$ and the square  $\widehat \mcq$ (shaded).
(b). The image $\Omega$ (resp. $\widehat \mct$ (shaded)) of the symmetric mapping $\widehat T$ whose  domain is $\mct_{\mathrm{up}}$  (resp. $\widehat \mcq$ (shaded)). (c).  Domains obtained from   $\widehat \mcq$ and the upper triangle $\mct_{\mathrm{up}}$ in (a) by the affine mapping $F_1.$
(d).  Domains obtained from   $\widehat \mct$ and $\Omega$ in (b) by the affine mapping $F_2.$ }
\label{symmapextfig}
\end{figure}

\section{Basis functions and computation of the stiffness matrix}\label{sect3}

We introduce in this section the modal and nodal basis functions on triangles, and
present a fast and accurate  algorithm for computing the stiffness  matrix with a focus on how to deal with
the singularity (cf. \eqref{observe}-\eqref{observe1}).

%We reiterate that the triangular element

\subsection{Modal basis}

Let $I=(-1,1)$ as before. We define the space
\begin{equation}\label{Ynsps}
Y_N(\mct)={\mathcal Q}_N(\mcq)\circ T^{-1}=(P_N(I))^2\circ T^{-1},
\end{equation}
which consists of the images of the tensor-product polynomials on $\mcq$ under the inverse mapping $T^{-1}$ defined in  \eqref{best_map_inv}. As a direct consequence of Proposition  \ref{tranfm} (ii), we have
\begin{equation}\label{Ynsps2}
Y_N(\mct)={\mathcal P}_N(\mct) \oplus \chi {\mathcal P}_{N-1}(\mct),
\end{equation}
where  $\chi=\sqrt{(x - y)^2 + 4(1 - x - y)},$  and we recall that ${\mathcal P}_N(\mct)$ is the set of polynomials on
$\mct$ of total degree at most $N.$ This implies that $Y_N(\mct)$ contains not only polynomials, but also special irrational functions: $\chi \phi$ for any $\phi\in {\mathcal P}_{N-1}(\mct).$
%This well characterizes the finite dimensional approximation space used for approximating functions on $\mct.$

%In the current subsection, we would like to  introduce the modal basis functions on $\mct,$ using the following irrational basis functions based on the new mapping \eqref{best_map}:
%\begin{equation}\label{duffybasisrational}
%\begin{split}
%	R_{mn}(x, y)&
%  =J_m^{0,0}(\xi) J^{0,0}_n(\eta)=\tilde R_{mn}(\xi,\eta).
%\end{split}
%\end{equation}
%Then we shall present an efficient and accurate means for evaluating  the associated stiffness and mass matrices.

Define
 \begin{equation}\label{basisc0}
 \begin{split}
&	\phi_0(\zeta)  = \frac{1 - \zeta} 2,\;\; \phi_k(\zeta)  = \frac{1 - \zeta^2} 4 J_{k - 1}^{1, 1}(\zeta),\;\; 1\le k\le N-1,\;\;
 \phi_N(\zeta) = \frac{1 + \zeta} 2,
\end{split}
\end{equation}
where $J_{k}^{1,1}$ is the Jacobi polynomial of degree $k$   (cf. \cite{szeg75}).  It is clear that   $\{\phi_k\}_{k=0}^N$
forms a basis of $P_N(I),$ and we have
\begin{equation}\label{tensorial}
{\mathcal Q}_N(\mcq)={\rm span}\big\{\Phi_{kl} : \Phi_{kl}(\xi,\eta)=\phi_k(\xi)\phi_l(\eta),\; 0\le k,l\le N  \big\}.
\end{equation}
It is a commonly used $C^0$-modal basis for QSEM, which  enjoys a distinct separation of the interior
and boundary modes (including vertex and edge modes). All interior modes are zero
on the triangle boundary. The vertex modes have a unit magnitude at one vertex and
are zero at all other vertices, and the edge modes only have magnitude along one edge
and are zero at all other vertices and edges.

In view of \eqref{Ynsps} and \eqref{tensorial}, we obtain the modal basis for $Y_N(\mct):$
\begin{equation}\label{tensoriassl}
Y_N(\mct)={\rm span}\big\{\Psi_{kl} : \Psi_{kl}(x,y)=\Phi_{kl}\circ T^{-1},\;  0\le k,l\le N  \big\}.
\end{equation}

%We start with a commonly used modal basis for $P_N(I),$  which consists of
%
%
%  It is clear that  $\{\phi_k\}_{k=0}^M$ forms a basis of $\CP_M(I).$ Then the modal basis for  $\Pi_M(\mcq)$  is formed by the tensor product:
%\begin{equation}\label{2dtensor}
%\Phi_{kl}(\xi,\eta)=\phi_k(\xi)\phi_l(\eta), \quad 0\le k,l\le M.
%\end{equation}
%
%Now, we define the modal basis on $\mct$ as
%\begin{equation}\label{2dtensor2}
%\Psi_{kl}(x,y)=\Phi_{kl}(\xi,\eta), \quad 0\le k,l\le M,
%\end{equation}
%where $(x,y)$ and $(\xi,\eta)$ are connected by the mapping \eqref{best_map}-\eqref{best_map_inv}. According, set
%\begin{equation}\label{spsonT}
%\CR_M:={\rm span}\big\{\Psi_{kl}(x,y)~:~ 0\le k,l\le M\big\},
%\end{equation}
%which is  obtained by mapping $\Pi_M(\mcq).$
% \comm{\color{red}\bf Shall we put some figures here to show the edge basis on $\mct$ as in the book of Karniadkas?}

\subsection{Computation of the stiffness matrix}\label{mathcomput} Though the singular integral of \eqref{observe1}-type  has a finite value, some efforts are needed to compute such integrals in a fast and stable manner.
Next, we devise an efficient algorithm for  this purpose.

Let $L_k$ be the Legendre polynomial of degree $k,$  and
recall that (see, e.g.,  \cite{szeg75})
\begin{align}
&(1-\zeta^2)J_{k - 1}^{1, 1}(\zeta)=\frac{2k}{2k + 1}\big(L_{k - 1}(\zeta) - L_{k + 1}(\zeta)\big), \label{Jacorec}\\
& (2k+1)L_k(\zeta)=L_{k+1}'(\zeta)-L_{k-1}'(\zeta), \label{deriv}  \\
& \zeta L_k(\zeta)= \frac{k}{2k+1}L_{k-1}(\zeta)+\frac{k+1}{2k+1} L_{k+1}(\zeta). \label{3term}
\end{align}
Thus, we have
\begin{equation}\label{convssum}
\phi_0'(\zeta)=-\frac 1 2 L_0(\zeta)=-\phi_N'(\zeta), \quad \phi_k'(\zeta)=-\frac k 2 L_k(\zeta),\;\; 1\le k\le N-1.
\end{equation}
By  \eqref{hndnt}, \eqref{grad_u_v}  and \eqref{tensorial}, %  that %\comm{\bf Need to change accordingly!}
\begin{equation}\label{evalumatrix}
\begin{split}
& \chi \partial_x \Psi_{kl} =%(3-\xi) \phi_k'(\xi)\phi_l(\eta)+(1+\eta) \phi_k(\xi)\phi'_l(\eta),
2\big(\phi_k'(\xi)\phi_l(\eta) + \phi_k(\xi)\phi'_l(\eta)\big) +\big [(1-\xi) \phi_k'(\xi)\phi_l(\eta)-(1-\eta) \phi_k(\xi)\phi'_l(\eta)\big],\\
& \chi \partial_y \Psi_{kl} =%(1+\xi) \phi_k'(\xi)\phi_l(\eta)+(3-\eta) \phi_k(\xi)\phi'_l(\eta).
2\big(\phi_k'(\xi)\phi_l(\eta) + \phi_k(\xi)\phi'_l(\eta)\big) -\big [(1-\xi) \phi_k'(\xi)\phi_l(\eta)-(1-\eta) \phi_k(\xi)\phi'_l(\eta)\big ].
\end{split}
\end{equation}
Thanks to \eqref{Jacorec}-\eqref{convssum},   $\chi \partial_x \Psi_{kl}$ and $\chi \partial_y \Psi_{kl}$ can be represented by a  linear combination of
$\{L_{k\pm i}(\xi)L_{l\pm j}(\eta) \}_{i,j=0,1}.$ In view of this, we can  evaluate  the entries of the stiffness matrix by computing  the integrals of the product of  Legendre polynomials:
\begin{equation}\label{stiffness}
s_{kl}^{k'l'}:=\iint_\mct \nabla \Psi_{kl}\cdot \nabla \Psi_{k'l'}\, \diff{x}\diff{y}\;\; \longleftrightarrow
\;\;  \iint_\mcq \frac{L_i (\xi)  L_{j}(\eta) L_{i'}(\xi)L_{j'}(\eta)} {2-\xi-\eta}\,  \diff{\xi}\diff{\eta}:=a_{ij}^{i'j'}.
\end{equation}
%which are all well-defined (cf. \eqref{observe1}).
Using the fact that  the product $L_m L_{n}$ can be represented  by $\{L_p\}_{p=0}^{m+n}$:
\begin{equation}\label{legprod}
L_m(\xi)L_n(\xi)=\sum_{p=0}^{m+n} c_p^{mn} L_p(\xi),
\end{equation}
where the expansion coefficient $\{c_p^{mn}\}$ can be found in  e.g., \cite{Hylleraas},  we obtain
\begin{equation}\label{ajikk}
a_{ij}^{i'j'} = \sum_{p = 0}^{i + i'}\sum_{q = 0}^{j + j'}c^{ii'}_p c^{jj'}_q\hat{a}_{pq}, \;\; {\rm where}\;\;
\hat{a}_{pq} = \iint_\mcq \frac{L_p (\xi)  L_q(\eta)} {2-\xi-\eta}\,  \diff{\xi}\diff{\eta}.
\end{equation}

Now, we describe how to compute $\{\hat a_{pq}\}$  in a fast and accurate manner.
This essentially relies on the following recurrence relation.
\begin{lem}\label{lem:rec} We have % the five-stencil relation:
\begin{equation}\label{firsttworecur}
\frac{\hat a_{p,q+1}-\hat a_{p,q-1}}{2q+1}=\frac{\hat a_{p+1,q}-\hat a_{p-1,q}}{2p+1}, \quad\forall\, p,q \ge 1.%,
\end{equation}
%where $\hat a_{p,-1}=0.$
\end{lem}
\begin{proof}
The statement is true for $p = q \ge 1,$ since $\hat a_{p, p\pm1}=\hat a_{p \pm1, p}$. In view of the symmetry: $\hat a_{pq}=\hat a_{qp},$
it suffices to show it holds for $p\ge q\ge 1.$

We start with recalling the Legendre functions of the second kind (see, Formula (4.61.4) in \cite{szeg75}):
\begin{align}\label{Qnx1}
 Q_n(x)  = \frac12 \int_{-1}^1   \frac{L_n(t)}{x-t}\, \diff{t}, \quad   n \ge 1; \quad  Q_0(x)=\frac 1 2 \ln \frac{x+1}{x-1}, \quad \forall x>1,
\end{align}
and the important identity (see \cite[Formula (4.62.1)]{szeg75}):
\begin{align}
 Q_n(x)  &= \frac 1 2 \Big(\ln \frac{x+1}{x-1}\Big) L_n(x)-\frac 1 2 \int_{-1}^1 \frac{L_n(x)-L_n(t)}{x-t}\,\diff{t}\nonumber\\
 &=\frac 1 2 \Big(\ln \frac{x+1}{x-1}\Big) L_n(x)-\tilde L_{n-1}(x). \label{Qnx2}
\end{align}
Here,  $\tilde L_n$ is the Legendre polynomial of the second kind, satisfying
\begin{align}\label{Lebseckind}
\tilde L_{n}(x)  = \frac{2n+1}{n+1} x \tilde  L_{n-1}(x) -\frac{n}{n+1} \tilde L_{n-2}(x), \;\; n \ge 1; \quad
\tilde L_{-1}(x) =  0, \;\;  \tilde L_0(x)=1,
\end{align}
which follows from \eqref{Qnx2} and  \cite[Formula (4.62.13)]{szeg75} directly.

%
%We now prove that it holds for $p>q\ge 1$. It is easy to verify that
%\begin{align*}
%%     Q_{n}(\eta)  =&\frac12 \int_{-1}^1   \frac{1}{2-\xi-\eta} \diff{\xi} \times L_n(2-\eta)
%%   - \frac12 \int_{-1}^1   \frac{L_n(2-\eta)-L_n(\xi)}{2-\xi-\eta} \diff{\xi}
%         \int_{-1}^1   \frac{L_q(\eta)}{2-\xi-\eta} \diff{\eta}  =& \int_{-1}^1   \frac{1}{2-\xi-\eta} \diff{\eta} \times L_q(2-\xi)
%   - \int_{-1}^1   \frac{L_q(2-\xi)-L_q(\eta)}{2-\xi-\eta}\, \diff{\eta}
%\\
%    = &%\frac12
%    \ln\frac{3-\xi}{1-\xi}  L_q(2-\xi)  -   2Q_{q-1}(2-\xi),\end{align*}
%where $Q_n$ is the Legendre function of  the second kind (cf. \cite{szeg75}), defined by %, %noting $\deg Q_n = n$.
%$$ Q_{n}(t)  =\frac12 \int_{-1}^1   \frac{L_n(\xi)}{t-\xi} \diff{\xi}, $$
%
%It is readily checked that
%\begin{align*}
%Q_k(\eta) = \sum_{0\le \nu \le k\atop k+\nu \text{ even}} \frac{4\nu+2}{(k-\nu+1)(k+\nu+2)} L_{\nu}(\eta).
%\end{align*}
Using \eqref{Qnx1}-\eqref{Lebseckind} and the orthogonality of the Legendre polynomials, we find that for  $p>q\ge 1$,
\begin{align}
\hat{a}_{pq}=&\int_{-1}^1 \int_{-1}^1 \frac{L_p(\xi) L_q(\eta) }{2-\xi-\eta}\,\diff{\xi}\diff{\eta}
= 2\int_{-1}^1 Q_q(2-\xi) L_p(\xi)\, \diff{\xi}\nonumber \\
= &
\int_{-1}^1\Big[ \Big(\ln\frac{3-\xi}{1-\xi}\Big)  L_q(2-\xi)  -   2\tilde L_{q-1}(2-\xi)\Big] L_p(\xi)\, \diff{\xi}\nonumber  \\
=&\int_{-1}^1 \Big(\ln\frac{3-\xi}{1-\xi}\Big)  L_q(2-\xi)   L_p(\xi)\,\diff{\xi}. \label{defncmb}
\end{align}
Thus, we have from \eqref{deriv} and integration by parts that
\begin{align*}
& \frac{\hat{a}_{p,q+1}-\hat{a}_{p,q-1}}{2q+1}=\int_{-1}^1 \Big(\ln \frac{3-\xi}{1-\xi}\Big) \frac{L_{q+1}(2-\xi)-L_{q-1}(2-\xi)}{2q+1}  L_p(\xi)\,\diff{\xi}\\
&\qquad =\int_{-1}^1 \Big(\ln \frac{3-\xi}{1-\xi}\Big)\frac{L_{q+1}(2-\xi)-L_{q-1}(2-\xi)}{2q+1} \Big[\frac{L_{p+1}(\xi)-L_{p-1}(\xi)}{2p+1}\Big]'\,\diff{\xi}\\
&\qquad =   -\int_{-1}^1  \left[\Big( \ln \frac{3-\xi}{1-\xi}\Big)    \frac{L_{q+1}(2-\xi)-L_{q-1}(2-\xi)}{2q+1}\right]'    \frac{L_{p+1}(\xi)-L_{p-1}(\xi)}{2p+1}\,\diff{\xi}.
\end{align*}
Working out the derivative, we obtain
\begin{align*}
& \frac{\hat{a}_{p,q+1}-\hat{a}_{p,q-1}}{2q+1}\\
&\overset{(\ref{deriv})} =  \int_{-1}^1  \left[ L_{q}(2-\xi)   \ln \Big(\frac{3-\xi}{1-\xi}\Big) -  \frac{L_{q+1}(2-\xi)-L_{q-1}(2-\xi)}{(q+1/2)(3-\xi)(1-\xi)}\right]  \frac{L_{p+1}(\xi)-L_{p-1}(\xi)}{2p+1}\, \diff{\xi}
\\
&\overset{(\ref{defncmb})} = \frac{\hat{a}_{p+1,q}-\hat{a}_{p-1,q}}{2p+1} -\int_{-1}^1  \frac{L_{q+1}(2-\xi)-L_{q-1}(2-\xi)}{(q+1/2)(3-\xi)(1-\xi)}
 \frac{L_{p+1}(\xi)-L_{p-1}(\xi)}{2p+1}\, \diff{\xi}\\
&\overset{(\ref{Jacorec})} = \frac{\hat{a}_{p+1,q}-\hat{a}_{p-1,q}}{2p+1}  +  \frac{1}{2pq} \int_{-1}^1   J^{1,1}_{q-1}(2-\xi) J^{1,1}_{p-1}(\xi) (1-\xi^2)\,\diff{\xi}\\
&~~= \frac{\hat{a}_{p+1,q}-\hat{a}_{p-1,q}}{2p+1},
%\\
\end{align*}
where we used the fact $p>q$ and the orthogonality of  Jacobi polynomials in the last step.
\end{proof}

Equipped with \eqref{firsttworecur}, we are able to compute $\{\hat a_{pq}\}_{p\ge q}$ accurately and rapidly.
We summarize the algorithm as follows.
\vspace{5mm} \hrule  \vspace{2mm} {\bf Algorithm for computing $\{\hat a_{pq}\}_{p,q=0}^N$}
\nolinebreak
\begin{itemize}
\item[1.] Initialization
\begin{itemize}
\item[(a)] For $p=0,1,\cdots, 2N,$ compute $\hat a_{p0}\,;$
\item[(b)] For $p=1,2,\cdots, 2N-1,$ compute $\hat a_{p1}\,.$
\end{itemize}

\item[2.]  For $q = 2, 3, \cdots, N,$\\
	For $p=q,\cdots, 2N - q,$
	\begin{equation}\label{itmeqn}
\hat{a}_{pq} = \hat a_{p,q-2} + \frac{2q-1}{2p+1}({\hat a_{p+1,q-1}-\hat a_{p-1,q-1}}),
\end{equation}
    Endfor of $p,q.$

\item[3.] Set $\hat a_{pq}=\hat a_{qp}$ for all $0\le p<q<N.$
\end{itemize}
\vspace{2mm} \hrule  \vspace{5mm}

We describe below the details for computing the initial values. %remarks are in order.
\begin{itemize}
  	\item  We find from \eqref{defncmb} that
	\begin{align}
		\hat{a}_{p0} & =  \int_{-1}^1 L_p(\xi)\ln\frac{3 - \xi}{1 - \xi}\,\diff{\xi}\nonumber\\
&= \int_{-1}^1 L_p(\xi)\ln\frac{3 - \xi}{2}\,\diff{\xi}+ \int_{-1}^1 L_p(\xi)\ln \frac 2{1 - \xi}\,\diff{\xi}
		        := \alpha_p + \beta_p.\label{ap0form}
	\end{align}
It is clear that by \eqref{deriv} and integration by parts,
	\begin{align*}
		\alpha_p & = \int_{-1}^1 L_p(\xi)\ln\frac{3 - \xi}{2}\,\diff{\xi}=
 \frac 1 {2p+1}\Big( \int_{-1}^1 \frac{L_{p+1}(\xi)}{3 - \xi}\,\diff{\xi}-\int_{-1}^1 \frac{L_{p-1}(\xi)}{3 - \xi}\,\diff{\xi} \Big).
 \end{align*}
 It decays exponentially with respect to $p,$ and the use of a Legendre-Gauss quadrature leads to an exponentially accurate approximation, since the function $1/(3-\xi)$ is analytic within an ellipse (see \cite{XWZh.11}).
 We find from e.g.,  \cite{Gau.W10} that
 \begin{align*}
\beta_p=\int_{-1}^1 L_p(\xi)\ln \frac 2{1 - \xi}\,\diff{\xi}=\begin{cases}
2,\quad &{\rm if}\;\; p=0,\\
\dfrac 2{p(p+1)},\quad &{\rm if}\;\; p\ge 1.
\end{cases}
 \end{align*}
 	\item  Using  \eqref{ajikk}, \eqref{3term} and the orthogonality of Legendre polynomials, we find
	\begin{align}
		\hat{a}_{p1} & = \iint_\mcq \frac{\eta L_p(\xi)}{2-\xi-\eta}\,\diff{\xi}\diff{\eta} = \iint_\mcq \frac{(2 - \xi) L_p(\xi)}{2-\xi-\eta}\,\diff{\xi}\diff{\eta} - \iint_\mcq \frac{(2 - \xi - \eta)L_p(\xi)}{2-\xi-\eta}\diff{\xi}\diff{\eta}\nonumber\\
		& = 2\iint_\mcq \frac{L_p(\xi)}{2-\xi-\eta}\,\diff{\xi}\diff{\eta} - \iint_\mcq \frac{\xi L_p(\xi)}{2-\xi-\eta}\,\diff{\xi}\diff{\eta} - \int_{-1}^{1}\left[\int_{-1}^1 L_p(\xi)\,\diff{\xi}\right]\diff{\eta}\nonumber\\
		& = 2\hat{a}_{p0} - \frac{(p + 1)\hat{a}_{p + 1, 0} + p\hat{a}_{p - 1, 0}}{2p + 1},\quad p\ge 1. \label{ap1form}
	\end{align}
\item We see that with an accurate computation of the initial values  $\{\hat a_{p0}\}$, marching by \eqref{ap1form} and \eqref{itmeqn} is expected to be stable.  In Figure \ref{stencils}, we provide  a schematic illustration of sweeping
    the stencils by the algorithm.
    \begin{figure}[th!]
\begin{minipage}[t]{0.7\textwidth}
\centering
		\includegraphics[width=1\textwidth]{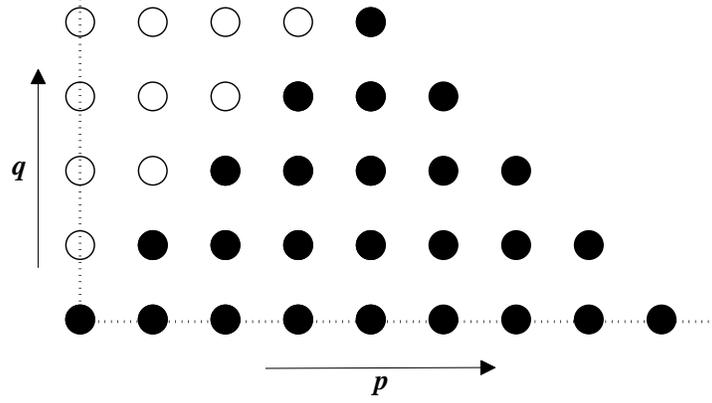}
\end{minipage}
	\caption{\small Diagram for computing $\{\hat a_{pq}\}_{p,q=0}^N$ with $N=2,$ where the stencils marked by ``$\bullet$" are marched via Steps 1-2 in the {\bf Algorithm}, and those marked by ``$\circ$"  are obtained by the symmetric property in Step 3.}
	\label{stencils}
\end{figure}
\end{itemize}

\begin{rem}\label{logrimap} We see that the rectangle-triangle mapping \eqref{best_map} essentially induces
logarithmic singularity. Indeed,  numerical quadrature of integrands involving a logarithmic weight function
is of independent interest (see, e.g., \cite{Gau.W10}).  \qed
\end{rem}

\begin{rem}\label{massm}
 As a quick note,  the mass matrix under this basis is sparse. Indeed, by \eqref{massmatrix},
\begin{align}\label{uvtsparse}
	(u, v)_{\mct} =  \frac{1}{8}\iint_\mcq \tilde{u}\tilde{v}\, \diff{\xi}\diff{\eta} - \frac{1}{16}\iint_\mcq \xi\tilde{u}\tilde{v}\,\diff{\xi}\diff{\eta} - \frac{1}{16}\iint_\mcq \eta\tilde{u}\tilde{v}\,\diff{\xi}\diff{\eta},
\end{align}
so we claim this from \eqref{3term} and the orthogonality of the Legendre polynomials. \qed
\end{rem}

\begin{rem}\label{genmaprem}  With an additional affine mapping, any triangular element $\mct_{\rm any}$ can be transformed to the reference square $\mcq.$   It is  important to point out that the stiffness and mass matrices on
$\mct_{\rm any}$ can be precomputed in a similar fashion as above. To justify this, we consider
 a general triangle $\mct_{\rm any}$ with vertices  $V_i = (x_i, y_i)$, $i = 1, 2, 3$.
Like \eqref{best_map},  we have the mapping from $\mcq$ to $\mct_{\rm any}:$
 \begin{equation}\label{xygenmap}
	(x, y) = (x_1, y_1)\frac{(1 - \xi)(1 - \eta)}{4} + (x_2, y_2)\frac{(1 + \xi)(3 - \eta)}{8} + (x_3, y_3)\frac{(3 - \xi)(1 + \eta)}{8},
%	(x, y) = P_1\frac{(1 - \xi)(1 - \eta)}{4} + P_2\frac{(1 + \xi)(3 - \eta)}{8} + P_3\frac{(3 - \xi)(1 + \eta)}{8}
	%, \quad \forall\, (\xi, \eta) \in {\mcq}.
\end{equation}
for all $(\xi, \eta) \in {\mcq}.$
A direct calculation leads to %gives %gives a mass matrix (cf. \eqref{uvtsparse})
\begin{align}\label{geninner}
	\inprod[\mct_{\rm any}]{u}{v} =  \frac{F}{8}\iint_\mcq \tilde{u}\tilde{v}\, \diff{\xi}\diff{\eta} - \frac{F}{16}\iint_\mcq \xi\tilde{u}\tilde{v}\,\diff{\xi}\diff{\eta} - \frac{F}{16}\iint_\mcq \eta\tilde{u}\tilde{v}\,\diff{\xi}\diff{\eta},
\end{align}
and
\begin{align}
	\inprod[\mct_{\rm any}]{\gradient{u}}{\gradient{v}} &= A\iint_\mcq\big(\widetilde{\nabla}\cdot\tilde{u}\big)\big(\widetilde{\nabla}\cdot\tilde{v}\big)\chi^{-1}\diff{\xi}\diff{\eta} + C\iint_\mcq\big(\widetilde {\nabla}^\intercal\tilde{u}\big)\big(\widetilde {\nabla}^\intercal\tilde{v}\big)\chi^{-1}\diff{\xi}\diff{\eta}\nonumber\\& -
B\iint_\mcq\big[\big(\widetilde{\nabla}\cdot\tilde{u}\big)(\widetilde {\nabla}^\intercal\tilde{v}\big) + \big(\widetilde {\nabla}^\intercal\tilde{u}\big)\big(\widetilde{\nabla}\cdot\tilde{v}\big)\big]\chi^{-1}\diff{\xi}\diff{\eta}, \label{geninnerg}
	 %\frac{1}{4}\iint_\mcq\big(\widehat{\nabla}^\intercal\cdot\tilde{u}\big)\big(\widehat{\nabla}^\intercal\cdot\tilde{v}\big)\chi^{-1}\diff{\xi}\diff{\eta}\\
\end{align}
where $\chi^{-1}=2/(2-\xi-\eta),$ the differential operators are defined in \eqref{hndnt}, and the constants are given by
\begin{align*}
       F & = (x_2 - x_1)(y_3 - y_1) - (x_3 - x_1)(y_2 - y_1)\not=0,\\
	A & = ((x_2 - x_3)^2 + (y_2 - y_3)^2) / (2F)% = |P_2 - P_3|^2 / 2F
	,\\
	B & = ((x_2 - x_1)^2 + (y_2 - y_1)^2 - (x_3 - x_1)^2 - (y_3 - y_1)^2) /(4F)% = (|P_2 - P_1|^2 - |P_3 - P_1|^2)/ 4F
	,\\
	C & = ((2x_1 - x_2 - x_3)^2 + (2y_1 - y_2 - y_3)^2) /(8F)% = |2P_1 - P_2 - P_3|^2 / 8F
	.
\end{align*}
In particular, if $\mct_{\rm any}=\mct,$  \eqref{geninner} and \eqref{geninnerg} (note: $B=0$)  reduce to
\eqref{massmatrix} and \eqref{grad_u_grad v}, respectively.

As with \eqref{evalumatrix}, we find from \eqref{Jacorec}-\eqref{convssum} that
 $\widetilde \nabla \cdot \Phi_{kl}$  and $\widetilde {\nabla}^\intercal \Phi_{kl}$ can be expressed in terms of
 $\{L_{k\pm i}(\xi)L_{l\pm j}(\eta) \}_{i,j=0,1},$ so the mass matrix on $\mct_{\rm any}$ can be precomputed by the same algorithm described above.   \qed
\end{rem}
%\comm{\bf Put the general transform!!!}

\subsection{Interpolation, quadrature and nodal basis}   Through the general mapping
\eqref{xygenmap}, the operations (e.g., interpolation, quadrature and numerical differentiations)  on  a triangular element can be performed on the reference square $\mcq.$

%In a quadrangle based spectral element method, it is advantageous  to use nodal basis functions, which are known  to be more convenient  to impose continuity across the elements and to deal with nonlinear problems.

Hereafter, let  $\{\zeta_j\}_{j=0}^N$ be the Legendre-Gauss-Lobatto (LGL) points, i.e., the zeros of $(1 - \zeta^2)L'_N(\zeta), $ and let   $\{h_j\}_{j=0}^N$ be the associated Lagrangian basis polynomials such that
$h_j\in P_N(I)$ and $h_j(\zeta_k) = \delta_{kj},$ where $\delta_{kj}$ is the Kronecker delta.
Given $v\in C(\bar I),$ the one-dimensional polynomial interpolant  of $u$ is
\begin{equation}\label{onedimeninterp}
(I_N^\zeta v)(\zeta)=\sum_{j=0}^N v(\zeta_j) h_j(\zeta)\in P_N,\quad \forall\, \zeta\in \bar I.
\end{equation}
Recall that the LGL quadrature has the exactness: % degree of precision $2N-1:$
\begin{equation}\label{oneexamc}
\int_{-1}^1 \phi(\zeta) d\zeta=\sum_{j=0}^N \phi(\zeta_j)\omega_j,\quad \forall\, \phi\in P_{2N-1}(I),
\end{equation}
where $\{\omega_j\}$ are the LGL quadrature weights.

Given any $u\in C(\bar \mct),$ we define the interpolant of $u$ by
\begin{equation}\label{inter2d}
(\II_N u)(x,y) = (I_N^\xi I_N^\eta \tilde u)\circ T^{-1}=\Big(\sum_{i,j=0}^n  (u\circ T)(\xi_i,\eta_j) h_i(\xi)h_j(\eta)\Big)\circ T^{-1},
\end{equation}
where  $T$ and $T^{-1}$ are defined in \eqref{best_map} and  \eqref{best_map_inv} as before, and $\{\xi_k=\eta_k=\zeta_k\}_{k=0}^N$.   Notice that  $\II_N u\in Y_N(\mct).$

We also extend the LGL quadrature to define the discrete inner product on $\mct$ as
\begin{equation}\label{disinner2d}
\langle u,v \rangle_{N,\mct}=\frac 1 8 \sum_{i,j=0}^N \tilde u(\xi_i,\eta_j) \tilde v(\xi_i,\eta_j)  \chi(\xi_i,\eta_j) \omega_i \omega_j,
\end{equation}
where $\chi=(2-\xi-\eta)/2.$  As a consequence of \eqref{massmatrix}, \eqref{oneexamc} and \eqref{Ynsps}-\eqref{Ynsps2}, there holds
\begin{equation}\label{disinner2dnew}
\langle u,v \rangle_{N,\mct}=(u,v)_{\mct},\quad \forall\, u\cdot v\in Y_{2N-2}(\mct),
\end{equation}
which also holds for all $u\cdot v\in {\mathcal P}_{2N-2}(\mct).$

Since $\{h_kh_l\}_{k,l=0}^N$ forms the nodal basis for  ${\mathcal Q}_N(\mcq),$ we can obtain the nodal basis for
$Y_N(\mct):$
\begin{equation}\label{nodebasist}
Y_N(\mct)={\rm span}\big\{\widehat \Psi_{kl} : \widehat \Psi_{kl}(x,y)=(h_kh_l)\circ T^{-1} \,:\,0\le k,l\le N  \big\}.
\end{equation}
In view of \eqref{uvtsparse}, the mass matrix under this nodal basis can be computed easily as usual by tensorial LGL quadrature.
However, the direct evaluation of the stiffness matrix like \eqref{stiffness} is prohibitive, as there is no recursive way for the computation. In order to surmount this obstacle,
we resort to the notion of ``discrete transform" (cf. \cite{ShenTaoWang2011}).  Like \eqref{evalumatrix}, we have
\begin{equation*}\label{evalumatrix2}
\begin{split}
& \chi \partial_x \widehat \Psi_{kl} =%(3-\xi) h_k'(\xi)h_l(\eta)+(1+\eta) h_k(\xi)h'_l(\eta)\\
2(h_k'(\xi)h_l(\eta)+h_k(\xi)h'_l(\eta)) + [(1-\xi) h_k'(\xi)h_l(\eta)-(1-\eta) h_k(\xi)h'_l(\eta)]\in {\mathcal Q}_N(\mcq),\\
& \chi \partial_y \widehat \Psi_{kl} =%(1+\xi) h_k'(\xi)h_l(\eta)+(3-\eta) h_k(\xi)h'_l(\eta)
2(h_k'(\xi)h_l(\eta)+h_k(\xi)h'_l(\eta)) - [(1-\xi) h_k'(\xi)h_l(\eta)-(1-\eta) h_k(\xi)h'_l(\eta)]\in {\mathcal Q}_N(\mcq).
\end{split}
\end{equation*}
The idea is to transform $\{\chi \partial_x \widehat \Psi_{kl}\}_{k,l=0}^N$ and $\{\chi \partial_y \widehat \Psi_{kl}\}_{k,l=0}^N$ to $\{L_i(\xi)L_j(\eta)\}_{i,j=0}^N$ via a two-dimensional  discrete transform. Then the evaluation boils down to finding $\{a_{ij}^{i'j'}\}$ in  \eqref{stiffness} as before.

\section{Estimates of orthogonal projection and interpolation errors}\label{sect:errorest}

The section is devoted to  error estimates of orthogonal projections and
 interpolations on triangles.  These results will be essential for understanding the approximability of the basis functions,
 and provide important tools for  error analysis of the TSEM for PDEs.

\subsection{Orthogonal projections}
We start with considering  the projection $\Pi_N: \Lspace[2]{\mct} \to Y_N(\mct),$ defined by
\begin{equation}\label{l2prj}
\big(\Pi_N u-u, v\big)_\mct = 0, \quad \forall\, v \in Y_N(\mct).
\end{equation}
\begin{thm}\label{ortL2err}
	For any  $u \in H^r(\mct)$ with $r\ge 0$, we have
	\begin{equation}\label{L2error}
		\|\Pi_N u-u\|_{\mct} \le cN^{-r}|u|_{r, \mct},
	\end{equation}
where $c$ is a positive constant independent of $N$ and $u.$
\end{thm}
\begin{proof} We have
\begin{equation}\label{prooftmp1}
\|\Pi_N u-u\|_{\mct}\overset{(\ref{l2prj})}=\inf_{\phi\in Y_N(\mct)} \|\phi -u\|_{\mct}
\overset{(\ref{Ynsps2})}\le \|\psi-u\|_{\mct},\quad \forall \psi\in {\mathcal P}_N(\mct).
\end{equation}
Now, we take $\psi$ to be the best $L^2$-approximation in ${\mathcal P}_N(\mct),$ denoted by $\pi_N u.$
By   \cite[Theorem 3.3]{LS10},
\begin{align}\label{polyappl2}
  \|\pi_N u-u\|_{\mct} \le c N^{-r} \Big(\sum_{k_1+k_2+k_3=r}
 \big\| \partial_x^{k_1}\partial_y^{k_2}(\partial_y - \partial_x)^{k_3}u\big\|^2_{\omega^{k_1,k_2,k_3},\mct}\Big)^{1/2}
 \le c N^{-r} |u|_{r,\mct},
  \end{align}
  where $\omega^{k_1,k_2,k_3}=x^{k_1+k_3} y^{k_2+k_3}(1-x-y)^{k_1+k_2}$ is a Jacobi weight function on $\mct.$
Therefore, the estimate \eqref{L2error} follows from \eqref{prooftmp1} and  \eqref{polyappl2}.
\end{proof}

We now turn to the $H^1$-projection: $\Pi_N^1: H^1(\mct)\to Y_N(\mct)$ such that
\begin{equation}\label{YnnY}
\big(\nabla (\Pi_N^1 u-u), \nabla v\big)_{\mct}+ \big(\Pi_N^1 u-u, v\big)_\mct=0,\quad \forall\, v\in Y_N(\mct),
\end{equation}
and  the $H^1_0$-projection:  $\Pi_N^{1,0}: H^1_0(\mct)\to Y_N^0(\mct)=Y_N(\mct)\cap H^1_0(\mct),$  defined by
\begin{equation}\label{YnnY2}
\big(\nabla (\Pi_N^{1,0} u-u), \nabla v\big)_{\mct}=0,\quad \forall\, v\in Y_N^0(\mct),
\end{equation}
where $H^1_0(\mct)$ is defined as usual, i.e., the subspace of $H^1(\mct)$ with functions vanishing on the boundary
of $\mct.$

\begin{thm}\label{ortH1err}
	For any $u \in H^1_0(\mct)\cap H^r (\mct)$ with $r \geq 1$, we have
	\begin{equation}\label{H1P}
		\|\Pi_N^{1,0} u-u\|_{\mu,\mct} \le c N^{\mu - r}  |u|_{r,\mct}, \quad \mu = 0, 1,
	\end{equation}
	where $c$ is a positive constant independent of $u$ and $N.$
	It  also holds for any $u\in H^r(\mct)$ with $\Pi_N^{1} u$ in place of $\Pi_N^{1,0} u.$
\end{thm}
\begin{proof}  Here, we only provide the proof for the projector $\Pi_N^{1,0},$ as the estimate for $\Pi_N^1$ can be obtained in a very similar fashion.

  By the Poincar%$\acute{e}$
\'{e} inequality, we know that the semi-norm  $| \cdot|_{1,\mct}$ is a norm of $H^1_0(\mct).$
  Hence, by the definition \eqref{YnnY2},
  \begin{equation}\label{mu1}
  	\|u - \Pi_{N}^{1,0} \|_{1,\mct} \le c |\phi-u|_{1,\mct}\le c \|\phi-u\|_{1,\mct},\quad \forall\, \phi\in Y_N(\mct).
  	\end{equation}
It is  known  from \eqref{Ynsps2} that ${\mathcal P}_N(\mct)\subset Y_N(\mct),$ so we can take $\phi$ to be the orthogonal
projection $\pi_N^{1,0}: H^1_0(\mct)\to {\mathcal P}^0_N(\mct)={\mathcal P}_N(\mct)\cap H^1_0(\mct),$ defined by
	\begin{equation}\label{defnP0}
	\big(\nabla (\pi_N^{1,0} u-u), \nabla v\big)_{\mct}=0,\quad \forall\, v\in {\mathcal P}_N^0(\mct).
	\end{equation}
We quote the estimate in 	 \cite[Theorem 3.4]{LS10}:
\begin{equation}\label{estmmLi}
\begin{split}
\|\pi_N^{1,0}u-u\|_{1,\mct} & \le c N^{1-r}  \Big(\sum_{k_1+k_2+k_3=r} \big\|\partial_x^{k_1}  \partial_y^{k_2} (\partial_y-\partial_x)^{k_3} u\big\|^2_{\omega^{k_1,k_2,k_3}_{+},\mct} \Big)^{1/2} \\
& \le c N^{1-r}|u|_{r,\mct},
\end{split}
\end{equation}
where $$\omega^{k_1,k_2,k_3}_{+}=x^{\max(k_1+k_3-1,0)}y^{\max(k_2+k_3-1,0)}(1-x-y)^{\max(k_1+k_2-1,0)}.$$
Hence, the estimate \eqref{H1P} with $\mu=1$ follows from \eqref{mu1} and \eqref{estmmLi}.

    To show  \eqref{H1P} 	with  $\mu = 0$,  we use a duality argument as in \cite{ciarlet1978finite}, which we sketch below.
    Given  $g \in \Lspace[2]{\mct},$  we consider the  auxiliary problem:
    Find  $u_g\in H^1_0(\mct)$ such that
  \begin{equation}\label{aux1}
  	a(u_g,v):= (\nabla u_g, \nabla v)_\mct = (g, v)_{\mct}, \quad \forall\, v \in  H^1_0(\mct).
	\end{equation}
 By a standard argument, we can show that this problem has a unique solution with the regularity
 $\|u_g\|_{2,\mct}\le c \|g\|_{\mct}.$
 	
	Now, taking $v =u - \Pi_N^{1,0} u$ into \eqref{aux1},  we find from  \eqref{YnnY2} and \eqref{H1P} with $\mu=1$ that
	\begin{align*}
		\big|(g, u-\Pi_N^{1,0} u)_\mct\big| & = \big| a(u_g,u - \Pi_{N}^{1,0} u)\big|=\big| a(u_g- \Pi_{N}^{1,0} u_g,u - \Pi_{N}^{1,0} u)\big|\\
		     & \leq \big |u_g - \Pi_{N}^{1,0} u_g\big|_{1,\mct} \big |u - \Pi_{N}^{1,0} u\big|_{1,\mct}  \\
		     &	\le c N^{-r} |u_g|_{2,\mct} |u|_{r,\mct}\le 	 c N^{-r} \|g\|_{\mct} |u|_{r,\mct}.
	\end{align*}
	Finally, we  derive
	\begin{equation*}
		\big\|u - \Pi_{N}^{1,0} u\big\|_{\mct} = \sup_{0 \neq g \in \Lspace[2]{\mct}}
		\frac{\big|(g, u-\Pi_N^{1,0} u)_\mct\big|}{\|g\|_{\mct}} \le c N^{-r} |u|_{r,\mct}.
	\end{equation*}
	This completes the proof.
\end{proof}

\begin{rem}\label{estrem1}  It is seen that benefited from the fact that  ${\mathcal P}_N(\mct)\subset Y_N(\mct),$ we are able to obtain the optimal error estimates directly from the available polynomial approximation results
on triangles.    \qed
\end{rem}

\subsection{Estimation of interpolation error}\label{intersec}

Now, we estimate the error of interpolation by \eqref{inter2d} on $\mct$.
The estimate of the  one-dimensional  LGL interpolation (cf. \eqref{onedimeninterp}) is useful for our analysis (see
\cite[Theorem 3.44]{ShenTaoWang2011}), that is,  for any $v\in H^r(I)$ with $r\ge 1,$  we have
\begin{equation}\label{onederror}
\big\|I_N^\zeta v-v\big\|_{L^2(I)}\le cN^{-r} \big\|(1-\zeta^2)^{(r-1)/2} v^{(r)}\big\|_{L^2(I)}.
\end{equation}
\begin{thm}\label{interr} For any $u \in H^r(\mct)$ with $r\ge 2$,
	\begin{align}\label{eqinterr}
\|\II_N u-u\|_{\mct}\le c N^{-r} B_r(u),
  \end{align}
  where
  \begin{equation}\label{Bru}
  B_r(u)=\begin{cases}
   |u|_{2, \mct} + \|(\partial_y - \partial_x)^2u\|_{J^{-1}, \mct} +\|\gradient{\cdot u}\|_{J^{-1}, \mct},\quad & {\rm if}\;\; r=2,\\
   |u|_{r,\mct} + |u|_{r-1,\mct},\quad & {\rm if}\;\; r\ge 3,
  \end{cases}
  \end{equation}
  $J$ is the Jacobian as defined in \eqref{Jacobiannew}, and  $c$ is a constant independent of $u$ and $N$.
\end{thm}
\begin{proof} To this end, let ${I_d}$ be the identity operator.  Using \eqref{massmatrix}, \eqref{inter2d} and  \eqref{onederror},
we obtain
	\begin{align*}
		\big\|\II_N u - u\big\|_{\mct}& \le c \big\|I_N^\xi I_N^\eta \tilde u - \tilde  u\big\|_{\mcq}
		= c\big\|(I_N^\xi - I_d)(I_N^\eta - I_d)\tilde u + (I_N^\xi - I_d)\tilde u + (I_N^\eta - I_d)\tilde{u}\big\|_{\mcq} \\
		                                  & \leq c \big(\big\|(I_N^\xi - I_d)(I_N^\eta - I_d)\tilde u\big\|_{\mcq} +
\big\|(I_N^\xi - I_d)\tilde u \big\|_{\mcq} + \big\|(I_N^\eta - I_d)\tilde u\big\|_{\mcq}\big)\\
                     & \le c N^{-1}\big\| (I_N^\eta - I_d) \partial_{\xi}\tilde{u} \big\|_{\mcq}
                      +c\big( \big\|(I_N^\xi - I_d)\tilde u \big\|_{\mcq} + \big\|(I_N^\eta - I_d)\tilde u\big\|_{\mcq}\big)\\
		          &\le  c N^{-r}\big(\big\|(1-\eta^2)^{(r-2)/2}\partial_{\xi} \partial_{\eta}^{r- 1}  \tilde{u}\big\|_{\mcq}
+\big\|(1-\xi^2)^{(r-1)/2}\partial_{\xi}^r\tilde{u}\big\|_{\mcq}\\
&\qquad  +
\big\|(1-\eta^2)^{(r-1)/2}\partial_{\eta}^r\tilde{u}\big\|_{\mcq}\big).
	\end{align*}
{It remains to transform the variables $(\xi,\eta)$ back to $(x,y)$ and obtain tight upper bounds of the right-hand side using  norms of $u$ on $\mct$.
%Herein, assume $u \in \Sspace[\varrho]{\mct}$, $\tilde{u}(\xi, \eta) = u(\T_{1 / 2}(\xi, \eta)) = u(x, y)$.  For brevity, the notation of \eqref{triL2seminorm} and \eqref{triH10seminorm} will be used.  Also, let %$\varpi^{\alpha, \beta}(\zeta) := (1 - \zeta)^\alpha(1 + \zeta)^\beta$, the Jacobi weight on the variable $\zeta$, sometimes denoted as $\varpi^{\alpha, \beta}_\zeta$.
%$\varpi^{\alpha, \beta}_\zeta := \varpi^{\alpha, \beta}(\zeta)$.
By \eqref{xy_xieta_inv},
\begin{align}
	\partial_\xi\tilde{u} & =
  \frac{1 - \eta}{4}\partial_y u - \frac{3 - \eta}{8}(\partial_y - \partial_x)u=	\frac{1 - \eta}{4}\partial_x u - \frac{1 + \eta}{8}(\partial_y - \partial_x)u,\label{xy_xieta_inv1}\\
	\partial_\eta\tilde{u} & = %\frac{-(1 + \xi)\partial_x u + (3 - \xi)\partial_y u}{8} =
	\frac{1 - \xi}{4}\partial_x u + \frac{3 - \xi}{8}(\partial_y - \partial_x)u = \frac{1 - \xi}{4}\partial_y u + \frac{1 + \xi}{8}(\partial_y - \partial_x)u.\label{xy_xieta_inv2}
\end{align}
Thus, we have
	\begin{equation}\label{xirhoderiv}
		\partial^r_\xi\tilde{u} = \sum\limits_{k = 0}^r (-1)^k\binom{r}{k}\left(\frac{1 + \eta}{8}\right)^k\left(\frac{1 - \eta}{4}\right)^{r - k}\partial^{r - k}_x(\partial_y - \partial_x)^k u,
	\end{equation}
and
\begin{align*}
		\big\|(1-\xi^2)^{(r-1)/2}&\partial_{\xi}^r\tilde{u}\big\|_{\mcq}^2= \iint_\mcq |\partial_{\xi}^r\tilde{u}|^2(1-\xi^2)^{r-1}\diff{\xi}\diff{\eta}\\%\frac{\diff{x}\diff{y}}{J}\\
%		& = \iint_\mct \left|\sum\limits_{k = 0}^r (-1)^k\binom{r}{k}\left(\frac{1 + \eta}{8}\right)^k\left(\frac{1 - \eta}{4}\right)^{r - k}\partial^{r - k}_x(\partial_x - \partial_y)^k u\right|^2\frac{(1 - \xi^2)^{r - 1}}{J}\diff{x}\diff{y}\\
		& \leq c \sum\limits_{k = 0}^r \iint_\mct\big|\partial^{r - k}_x(\partial_y - \partial_x)^k u\big|^2
 \frac{Q(\xi,\eta;r,k)}{J}\,\diff{x}\diff{y},
	\end{align*}
where
\begin{align*}
Q(\xi,\eta;r,k)=\Big(\frac{1 + \eta}{8}\Big)^{2k}\Big(\frac{1 - \eta}{4}\Big)^{2r - 2k} (1-\xi^2)^{r-1}.
\end{align*}
One verifies readily from  \eqref{best_map} that
\begin{align}
 & \frac  1 4 (1 - \xi)(1 - \eta) = 1 - x - y,\label{rela1}\\
&		\frac 1 4 {(1 + \xi)(1 - \eta)}+ \frac  1  8 {(1 + \xi)(1 + \eta)} = x,\label{rela2} \\
&	\frac 1 4 {(1 - \xi)(1 + \eta)}+ \frac 1  8 {(1 + \xi)(1 + \eta)}= y.\label{wtbd}
\end{align}
Therefore, by \eqref{rela1}-\eqref{wtbd},  we derive that for $2\le r\le k-1,$
\begin{align*}
Q(\xi,\eta;r,k)&= \frac 1 {2^k}\Big[(1+\xi)^k\Big(\frac {1+\eta} 8   \Big)^k \Big]\Big[(1-\xi)^k\Big(\frac {1+\eta} 4   \Big)^k \Big] \\
&\quad\times \Big(\frac {(1+\xi)(1-\eta)} 4   \Big)^{r-k-1} \Big(\frac {(1-\xi)(1-\eta)} 4   \Big)^{r-k-1}  \frac {(1-\eta)^2} {16} \\
&\le c x^k y^k x^{r-k-1} (1-x-y)^{r-k-1} J^2	\le c \varpi^{r-1,k,r-k-1}J,
%	\frac{1 - \eta}{16}&\frac{(1 + \eta)^k(1 - \eta)^{r - k - 1}(1 + \xi)^{r - 1}}{8^k4^{r - k - 1}}\frac{(1 + \eta)^k(1 - \xi)^k}{4^k}\frac{(1 - \eta)^{r - k - 1}(1 - \xi)^{r - k - 1}}{4^{r - k - 1}}\\&= \frac{2^k}{4}\left(\frac{1 + \eta}{8}\right)^{2k}\left(\frac{1 - \eta}{4}\right)^{2r - 2k - 1}(1 - \xi^2)^{r - 1} \leq \omega^{r-k-1,0,k}J;
	\end{align*}
where we  used the fact:  $1 - \eta \leq 2 - \xi - \eta = 16 J, $ and denoted by $\varpi^{\alpha,\beta,\gamma}=x^\alpha y^\beta (1-x-y)^\gamma.$ Similarly,  for $2\le r=k,$
\begin{align*}
	Q(\xi,\eta;r,k)&=\frac 1 {2^r}\Big(\frac {(1+\xi)(1+\eta)} 8 \Big)^{r-1} \Big(\frac {(1-\xi)(1+\eta)} 4 \Big)^{r-2}
	\Big(\frac {(1+\eta)} 4 \Big)^{2} (1-\xi)
	\\& \le   c  x^{r-1} y^{r-2} J\le c \varpi^{r-1,r-2,0}J,
%	\frac{(1 + \eta)^{r - 2}(1 + \xi)^{r - 2}}{8^{r - 2}}\frac{(1 + \eta)^{r - 2}(1 - \xi)^{r - 2}}{4^{r - 2}}\frac{1 - \xi}{16} = \frac{2^r}{64}\left(\frac{1 + \eta}{8}\right)^{2r - 4}(1 - \xi^2)^{r - 2}(1 - \xi)\\& \leq \omega^{0,0,r-2}J;
	\end{align*}
where we  used  $1 - \xi \leq 2 - \xi - \eta = 16 J. $ Consequently,  we obtain for $r\ge 2,$
\begin{equation}\label{xirnorm}
\begin{split}
		&\big\|(1-\xi^2)^{(r-1)/2}\partial_{\xi}^r\tilde{u}\big\|_{\mcq} \\
		&\quad \leq c\Big(\sum_{k = 0}^{r - 1}\big\|\partial_{x}^{r-k}(\partial_y - \partial_x)^k u\big\|_{\varpi^{r-1,k,r-k-1}, \mct}^2
		 + \big\|(\partial_y - \partial_x)^r u\big\|_{\varpi^{r-1,r-2,0}, \mct}^2\Big)^{\frac{1}{2}}\\
		 &\quad \le  c \big(|u|_{r-1,\mct}+|u|_{r,\mct}\big).
	\end{split}
	\end{equation}
By swapping $x \leftrightarrow y$ and $\xi \leftrightarrow \eta$, we get that for $r\ge 2,$
	\begin{equation}\label{etarnorm}
		\big\|(1-\eta^2)^{(r-1)/2}\partial_{\eta}^r\tilde{u}\big\|_{\mcq} \leq  c \big(|u|_{r-1,\mct}+|u|_{r,\mct}\big).
	\end{equation}

We now turn to deal with the term  $\big\|(1-\xi^2)^{(r-2)/2}\partial_\eta\partial_{\xi}^{r-1}\tilde{u}\big\|_{\mcq}$. By
\eqref{xy_xieta_inv2} and  \eqref{xirhoderiv},
	\begin{align}
		%\psub[\varrho - 1]{\psub{\tilde{u}}{\eta}}{\xi} & =
		&\partial_\eta\partial^{r - 1}_\xi\tilde{u}
		 = \partial_\eta\left[\sum\limits_{k = 0}^{r - 1} (-1)^k\binom{r - 1}{k}\left(\frac{1 + \eta}{8}\right)^k\left(\frac{1 - \eta}{4}\right)^{r - k - 1}\partial^{r - k - 1}_x(\partial_y - \partial_x)^k u\right]\nonumber\\
		&\quad = \sum_{k = 0}^r W_1(\xi, \eta;r,k)\partial^{r - k}_x(\partial_y - \partial_x)^k u +
		 \sum_{k = 0}^{r - 1} W_2(\xi, \eta;r,k)\partial^{r - k - 1}_x(\partial_y - \partial_x)^k u, \label{estwes}
	\end{align}
		where $W_1$ and $W_2$ are polynomials of $\xi$ and $\eta.$
%	
%	
%	\begin{align*}
%		f_0(\xi, \eta) & = (1 - \eta)^{r - 1}(1 - \xi) / 4^r,\qquad	f_r(\xi, \eta) = (-1)^{r - 1}(1 + \eta)^{r - 1}(3 - \xi) / 8^r, \text{ else }\\
%		f_k(\xi, \eta) & = \frac{(-1)^k(r - 1)}{16}\binom{r - 2}{k - 1}\left(\frac{1 + \eta}{8}\right)^{k - 1}\left(\frac{1 - \eta}{4}\right)^{r - k - 1}\left(\frac{(3 - \xi)(1 - \eta)}{2(r - k)} - \frac{1 - \xi}{k}\right),\\
%		g_0(\xi, \eta) & = (1 - r)(1 - \eta)^{r - 2} / 4^{r - 1},\qquad g_{r - 1}(\xi, \eta) = (-1)^{r - 1}(r - 1)(1 + \eta)^{r - 2} / 8^{r - 1}, \text{ else }\\
%		g_k(\xi, \eta) & = \frac{(-1)^k(r - 1)}{16}\binom{r - 1}{k}\left(\frac{1 + \eta}{8}\right)^{k - 1}\left(\frac{1 - \eta}{4}\right)^{r - k - 2}\left(k - \frac{(r - 1)(1 + \eta)}{2}\right).
%	\end{align*}
	Thus,  we have %similar to the earlier approach,
	\begin{align*}
		\big\|(1-\xi^2)^{(r-2)/2}\partial_\eta\partial_{\xi}^{r-1}\tilde{u}\big\|_{\mcq}^2 & \leq  c \sum_{k = 0}^r \iint_\mct \left|\partial^{r - k}_x(\partial_y - \partial_x)^k u\right|^2 \frac{%W_1(\xi, \eta;r,k)^2
		(1 - \xi^2)^{r - 2}}{J}\,\diff{x}\diff{y}\\ & + c \sum_{k = 0}^{r - 1} \iint_\mct \left|\partial^{r - k - 1}_x(\partial_y - \partial_x)^k u\right|^2 \frac{%W_2(\xi, \eta;r,k)^2
		(1 - \xi^2)^{r - 2}}{J}\,\diff{x}\diff{y}.
	\end{align*}
This implies that for $r\ge 3,$
		\begin{equation}\label{etarnormw}
		\big\|(1-\xi^2)^{(r-2)/2}\partial_\eta\partial_{\xi}^{r-1}\tilde{u}\big\|_{\mcq} \leq  c \big(|u|_{r-1,\mct}+|u|_{r,\mct}\big).
	\end{equation}
For $r=2,$ we obtain from a direct calculation that
 \begin{equation}\label{newqn}
 \begin{split}
 \|\partial_\xi\partial_\eta\tilde{u}\|_\mcq &\leq |u|_{2, \mct} +\frac 1 {256} \|(\partial_y - \partial_x)^2u\|_{J^{-1}, \mct} +
 \frac 1 {64} \|\gradient{\cdot u}\|_{J^{-1}, \mct}.
 \end{split}
 \end{equation}
 A combination of \eqref{xirnorm}-\eqref{etarnorm} and \eqref{etarnormw}-\eqref{newqn} leads to the desired result.
}\end{proof}	

\begin{rem}\label{estmrem} Like \eqref{xirnorm}, we could obtain sharper estimates with semi-norms in the upper bound of
\eqref{eqinterr} featured with the Jacobi-type weight functions $\varpi^{\alpha,\beta,\gamma}.$

Notice that for $r=2,$ the semi-norms are weighted with $J^{-1},$  as we can not factor out $1-\xi$ or $1-\eta$ from $W_1$ and $W_2$  in \eqref{estwes} to eliminate $J^{-1}$.  However, we point out that the value of $\int\hspace*{-5pt}\int_{\mct} J^{-1}\diff{x}\diff{y} $ is finite. \qed
\end{rem}

\section{Numerical results and concluding  remarks}

In this section, we just provide some numerical results to demonstrate the high accuracy
   of the proposed algorithm for model elliptic problems on $\mct.$  We also intend to compare it with the standard tensor-product spectral approximations on rectangles to assess the performance of our approach.

Consider the  elliptic equation:
\begin{equation}\label{numtest2}
\begin{split}
& -\Delta u + \gamma  u = f,\quad  \text{ in } \mct\,,\quad  u|_{\Gamma_1}= 0, \quad \p{u}{\nu}\Big|_{\Gamma_2} = g,
\end{split}
\end{equation}
where the constant $\gamma\ge 0,$  $\Gamma_1$ is the edges $x=0$ and $y=0$,
$\Gamma_2$ is the hypotenuse of $\mct,$    and ${\nu}$ is the unit vector outer normal to $\Gamma_2.$

\subsection{The scheme and its convergence}
A weak formulation of \eqref{numtest2} is to find $u\in H^1_{\Gamma_1}(\mct):=\big\{u\in H^1(\mct) : u|_{\Gamma_1}=0\big\}$ such that
\begin{equation}\label{numtest2w}
\begin{split}
& {\mathcal B}(u,v):=(\gradient{u}, \nabla v)_{\mct}+ \gamma (u, v)_{\mct} = (f,v)_{\mct}+\gamma
\langle g, v\rangle_{\Gamma_2},\quad \forall\, v\in H^1_{\Gamma_1}(\mct),
\end{split}
\end{equation}
where $\langle \cdot,\cdot \rangle_{\Gamma_2}$ is the inner product of $L^2(\Gamma_2).$
It follows from  a  standard argument that if $f\in L^2(\mct)$ and $g\in L^2(\Gamma_2),$ the problem \eqref{numtest2w} admits a unique solution in $H^1_{\Gamma_1}(\mct)$.
%\begin{align}
%\label{wellposed}
%    \|u\|_{1} \lesssim \|f\| +  \|g\|_{ \Gamma_2}.
%\end{align}

The spectral-Galerkin approximation  of \eqref{numtest2w}
is to find $u_N \in Y_N^{\Gamma_1}(\mct):=Y_N(\mct)\cap H^1_{\Gamma_1}(\mct)$ such that for any
$v_N\in Y_N^{\Gamma_1}(\mct),$
\begin{equation}\label{Gaform}
{\mathcal B}_N(u_N,v_N):=(\nabla u_N, \nabla
v_N)_{\mct} +\gamma (u_N,v_N)_{\mct}
=(\II_N f, v_N)_{\mct}+\langle g,v_N \rangle_{N,\Gamma_2},
\end{equation}
where $\II_N $ is the interpolation operator as defined in  \eqref{disinner2d}, and the discrete inner product
$\langle g,v_N \rangle_{N,\Gamma_2}$ can be defined on  the quadrature rule:  %\comm{\bf Mike: write the quadrature rule in two parts!}
\begin{equation}\label{boundterm}
\int_{\Gamma_2} g\, \diff{\gamma} = \frac{\sqrt{2}}{2}\Big[\int_{-1}^1\tilde{g}(\xi, 1)\diff{\xi}- \int_{-1}^{1}\tilde{g}(1, \eta)\diff{\eta}\Big] \sim  \frac 1{\sqrt{2}}\Big[\sum\limits_{j=0}^N \big(\tilde{g}(\zeta_j, 1) - \tilde{g}(1 , \zeta_j)\big)\omega_j\Big],
\end{equation}
where $\{\zeta_j,\omega_j\}$ are the LGL interpolation nodes and weights as  before. More precisely, we define
\begin{equation}\label{discinner}
\langle g,v_N \rangle_{N,\Gamma_2}= \frac {1}{\sqrt 2}
\sum_{j=0}^N  \tilde{g}(\zeta_j, 1)\tilde v_N(\zeta_j, 1)\omega_j  - \frac {1}{\sqrt 2}
\sum_{j=0}^N  \tilde{g}(1 , \zeta_j)\tilde v_N(1,\zeta_j)\omega_j,
\end{equation}
where $\tilde g=g\circ T$ and $\tilde v_N=v_N\circ T$.

\begin{rem}\label{ellipanal}
Here, we purposely impose the Neumann boundary condition on the hypotenuse of $\mct,$ so that
the basis functions associated with this ``singular" edge are involved in the computation.

We reiterate that a distinctive difference with the scheme in \cite[Eqn. (25)]{LWLM} lies in that
the consistency  condition \eqref{ptsingular} is not needed to be built in the approximation space,
which significantly facilitates the implementation.  Note that the approaches based on the Duffy's transform
also need to modify the basis functions to meet the corresponding consistency condition (see, e.g., \cite{SWL,SXL11}). \qed
\end{rem}

To analyze the convergence of  \eqref{Gaform}, it is essential to study the approximability of the orthogonal projection:
$\Pi_N^{1,\Gamma_1} : H^1_{\Gamma_1}(\mct)\to Y_N^{\Gamma_1}(\mct),$  such that
$$
\big(\nabla(\Pi_N^{1,\Gamma_1}u-u), \nabla \phi)_{\mct}=0,\quad \forall \phi \in  Y_N^{\Gamma_1}(\mct).
$$
Following the lines %as in
of the proof of Theorem \ref{ortH1err}, we find that  \eqref{H1P} holds with
$\Pi_N^{1,\Gamma_1}$ and $H^1_{\Gamma_1}(\mct)$ in place of  $\Pi_N^{1,0}$ and $H^1_{0}(\mct)$, respectively.

Another ingredient for the analysis  is to estimate the error  between the
continuous and discrete inner products on $\Gamma_2$.  Using  \cite[Lemma 4.8]{ShenTaoWang2011} leads to
\begin{align*}
\big|\langle g,v_N \rangle_{N,\Gamma_2}-\langle g,v_N \rangle_{\Gamma_2}\big|\le & c N^{-t}\Big(
\big\|(1-\xi^2)^{(t-1)/2} \partial_\xi^t \tilde g(\cdot, 1)\big\|_{L^2(I)}\|\tilde v_N(\cdot,1)\|_{L^2(I)}\\
&+  \big\|(1-\eta^2)^{(t-1)/2} \partial_\eta^t \tilde g(1,\cdot)\big\|_{L^2(I)} \|\tilde v_N(1,\cdot)\|_{L^2(I)}\Big).
\end{align*}
Then we obtain from \eqref{xy_xieta_inv1}-\eqref{xy_xieta_inv2} and a  derivation similar to the proof of
Theorem \ref{interr} the following estimate:
\begin{equation}\label{traceest}
\begin{split}
\big|\langle g,v_N \rangle_{N,\Gamma_2}-\langle g,v_N \rangle_{\Gamma_2}\big| &\le c N^{-t}\|(xy)^{(t-1)/2}(\partial_y - \partial_x)^tg\|_{\Gamma_2}\|v_N\|_{\Gamma_2}\\
&\le c N^{-t} |g|_{t, \Gamma_2}\|v_N\|_{\Gamma_2},\quad t\ge 1.
\end{split}
\end{equation}

With the above preparations, we can prove the convergence of the scheme \eqref{Gaform} by using
 Theorems \ref{ortH1err}-\ref{interr},  the estimate \eqref{traceest}  and a standard  argument for error estimate of spectral approximation of elliptic problems.
 \begin{thm}\label{th:convapp} Let $u$ and $u_N$ be the solutions of \eqref{numtest2w} and \eqref{Gaform}, respectively.  If
$u \in H^1_{\Gamma_1}
 (\mct ) \cap  H^{r}(\mct )$, $f \in H^s(\mct )$ and $g\in H^t(\Gamma_2)$  with $r\ge 1, s\ge 2$
and $t\ge 1$, then we have
\begin{align*}
\|u-u_N\|_{\mu, \mct} \le c \big(N^{\mu-r} |u|_{r, \mct} +  N^{-s}B_s(f)+ N^{-t} |g|_{t,\Gamma_2}\big),
\end{align*}
where $\mu=0,1,$ $B_s(f)$ is  defined in \eqref{Bru}, and $c$ is a positive constant independent of $N$ and
$u,f,g.$
\end{thm}

\subsection{Numerical results}

We first intend to show the typical spectral accuracy of the proposed method, so we particularly test it on
\eqref{numtest2} (with  $\gamma=1$) with the exact solution:
\begin{equation}\label{ex21}
	u(x, y) = e^{x + y - 1}\sin\big(3xy\big( y - {\sqrt{3}x}/{2} + {\sqrt{3}}/{4}\big)\big),\quad \forall\, (x,y)\in \mct.
\end{equation}
For comparison, we  also consider the standard tensor polynomial approximation of \eqref{numtest2} on a square $S=(0,1/\sqrt 2)^2$ (note: it has the same area as $\mct$) under a similar setting, i.e., Neumann data on two edges $x=1/\sqrt 2$ and $y=1/\sqrt 2,$ and homogeneous Dirichet data on the other two edges. We take the  exact solution:
\begin{equation}\label{ex21as}
	u(x, y) = {\rm exp}\Big(-\Big(\frac{1}{\sqrt{2}} - x\Big)\Big(\frac{1}{\sqrt{2}} - y\Big)\Big)\sin\big(3xy\big(y - {\sqrt{3}x}/{2} + {\sqrt{3}}/{4}\big)\big),\quad \forall (x,y)\in S.
\end{equation}

%\vskip 30pt
%{\bf Put a row of two figures here! Left: $L^2-$ and $L^\infty-$errors of modal basis! Right: Left: $L^2-$ and $L^\infty-$errors of nodal basis (for both triangle and rectangle)!}
\begin{figure}[!ht]
	\begin{minipage}{0.48\textwidth}
		\includegraphics[width=1\textwidth]{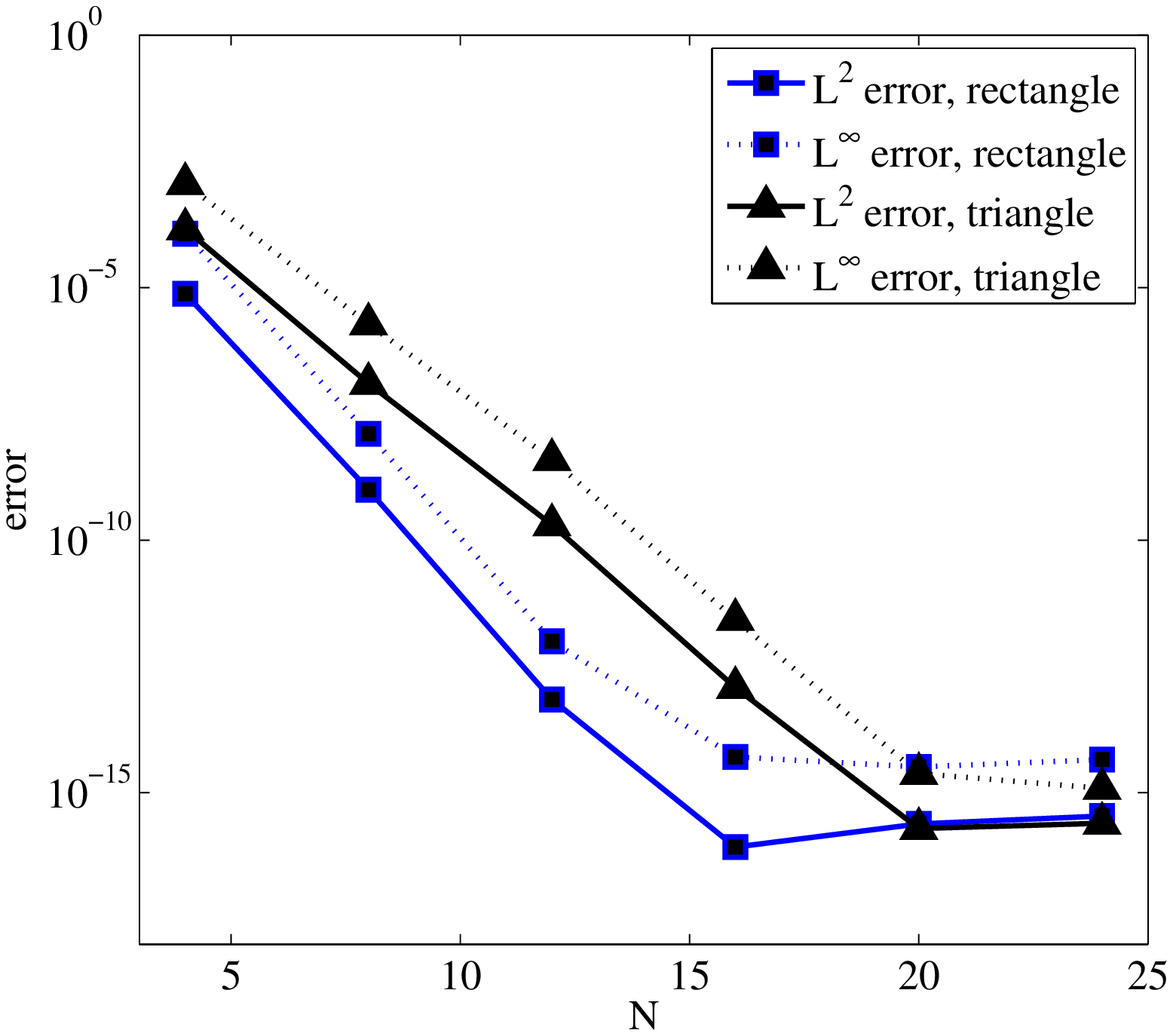}
	\end{minipage}\ %
	\begin{minipage}{0.48\textwidth}
		\includegraphics[width=1\textwidth]{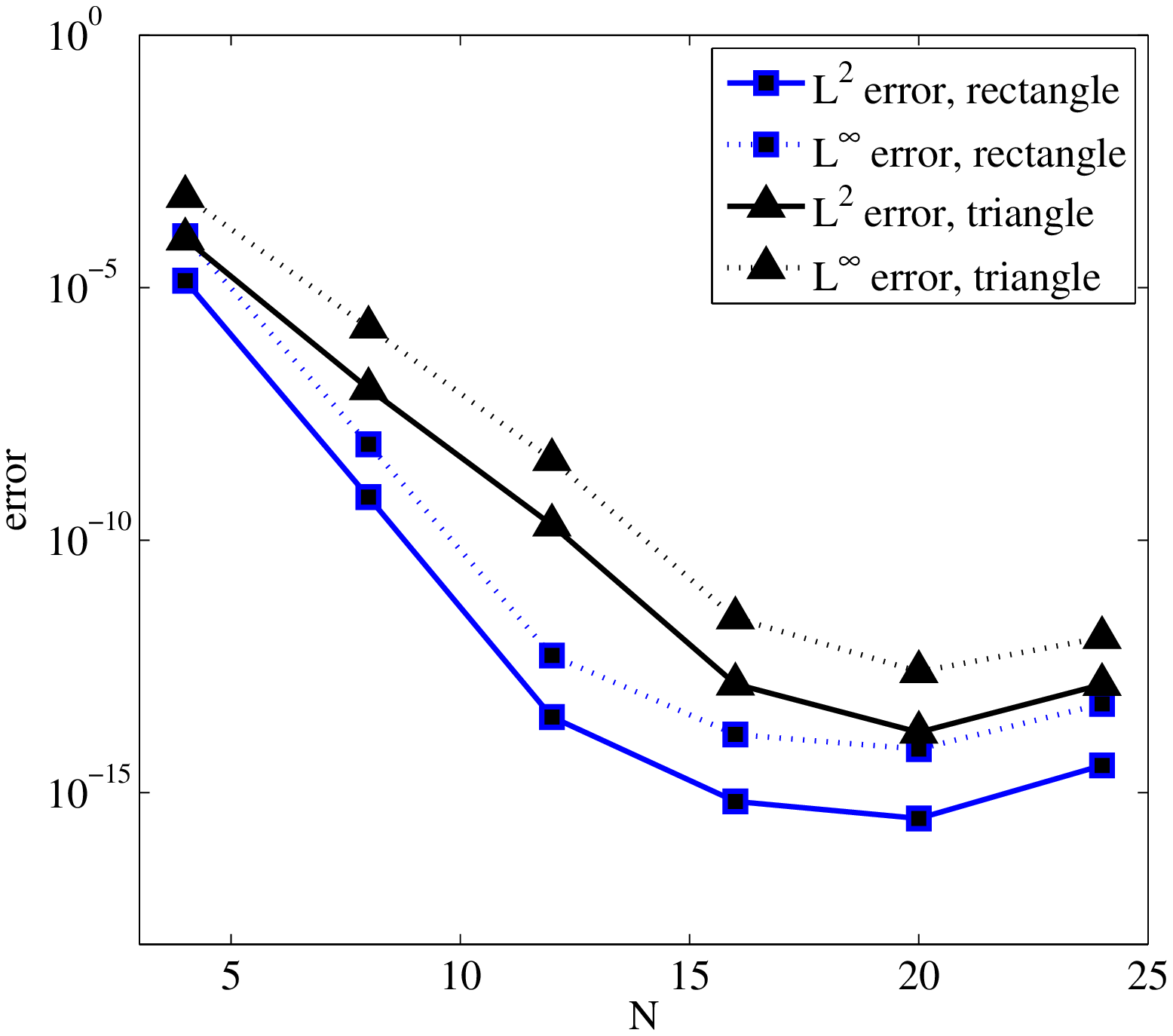}
	\end{minipage}\ %
	\caption{\small Numerical errors of (\ref{Gaform}) vs. tensorial polynomial approximation on the square $S$. Left: $L^2$- and $L^\infty$-errors using modal basis. Right: $L^2$- and $L^\infty$-errors using nodal basis.}
	\label{ex1figs}
\end{figure}

In  Figure \ref{ex1figs}, we plot the numerical errors of two methods, from which we observe that they share a very similar
convergence behavior and the errors decay like $O(e^{-cN}).$   For a fixed $N,$ the accuracy of approximation on $S$ seems to be slightly better than expected. We refer to  \cite[Figure 2.17]{CHQZ06b} for a similar comparison of the polynomial approximations on triangles \cite{Dubiner} and rectangles.     Indeed, the accuracy is comparable to the existing means in \cite{Kar.S05,SWL,LWLM}.

%the irrational spectral method on triangles shares %{\color{red}exhibit}
%the same asymptotic convergence and ``spectral'' accuracy as the polynomial spectral method on rectangles for smooth problems.

In the second test,  we choose the exact solution of \eqref{numtest2} with  finite regularity:
\begin{equation}\label{ex22}
	u(x, y) = (1 - x - y)^\frac{5}{2}(e^{xy} - 1),\quad \forall\,(x,y)\in \mct,
\end{equation}
which belongs to  $H^{3-\epsilon}(\mct)$ (for small $\epsilon>0$).
The counterpart on the square $S$ takes the form:
\begin{equation}\label{ex22f}
	u(x, y) = \Big(\frac 1 {\sqrt 2} - x\Big)^\frac{5}{2} \Big(\frac 1 {\sqrt 2} - y\Big)^\frac{5}{2}(e^{xy} - 1),\quad \forall\, (x,y)\in S.
\end{equation}

\begin{figure}[!ht]
   \hfill%
	\begin{minipage}{0.48\textwidth}
		\includegraphics[width=1\textwidth]{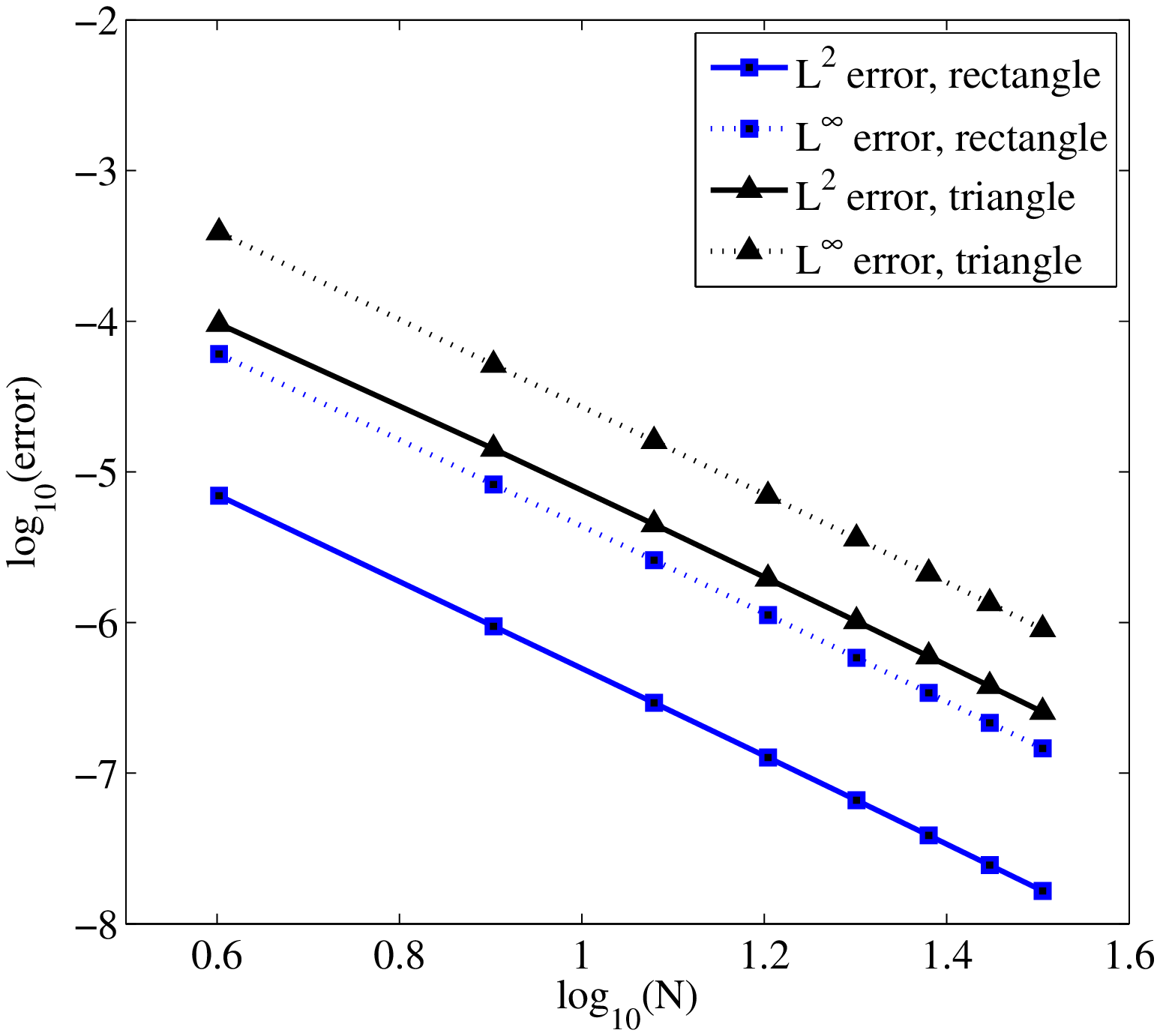}
	\end{minipage}   \hfill%
	\begin{minipage}{0.48\textwidth}
		\includegraphics[width=1\textwidth]{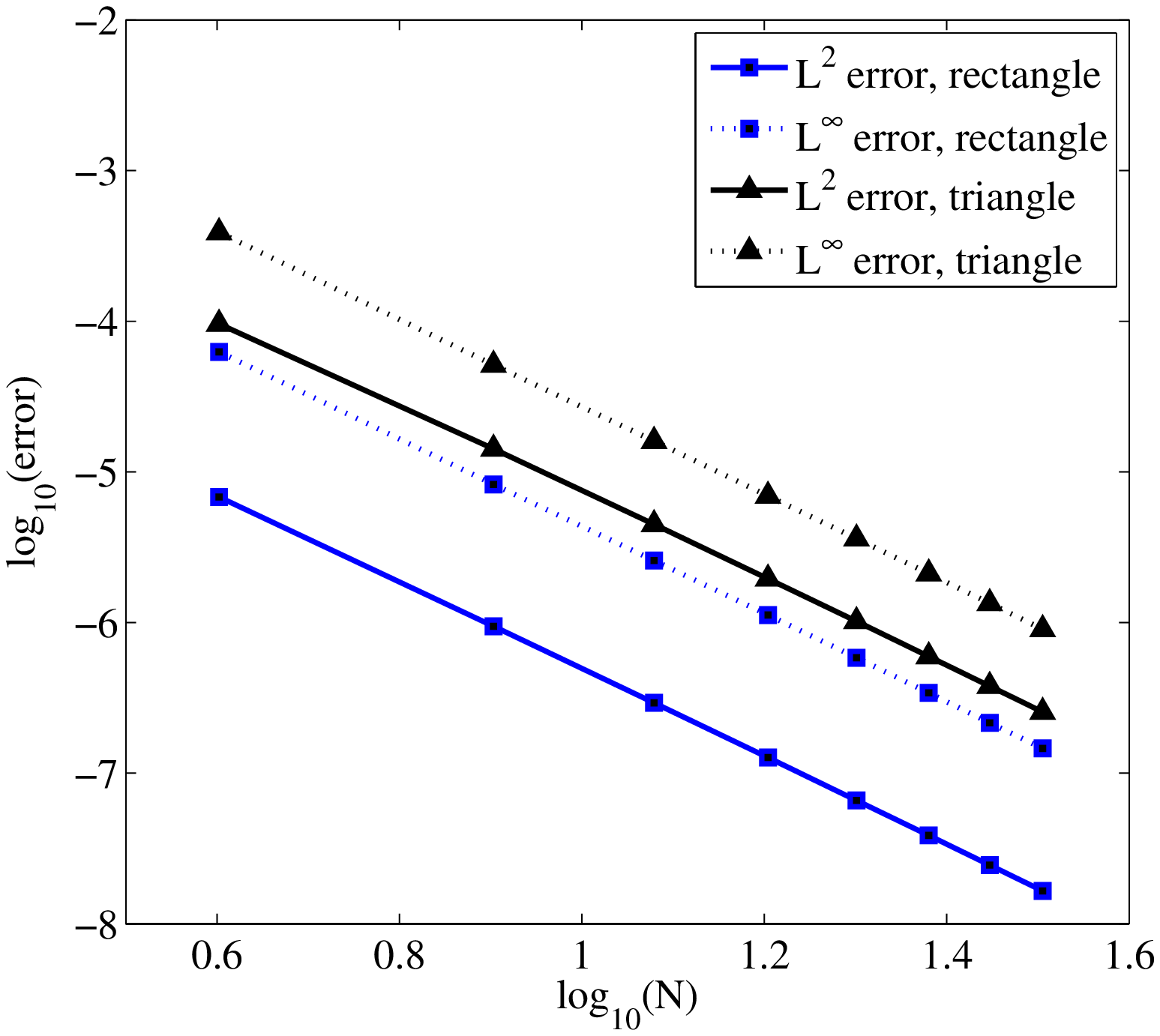}
	\end{minipage}\hspace*{\fill}
	\caption{\small \small Numerical errors of (\ref{Gaform}) vs. tensorial polynomial approximation on the square $S$ with solutions having finite regualrity.  Left: $L^2$- and $L^\infty$-errors using modal basis. Right: $L^2$- and $L^\infty$-errors using nodal basis.}
	\label{ex2figs}
\end{figure}
%\vskip 30pt

We depict in Figure \ref{ex2figs} the numerical errors of two approaches in $\log$-$\log$ scale, where the slopes of the lines  are all roughly $-3$ as predicted by the theoretical results (cf. Theorem \ref{th:convapp}).

%As %the problem
%{\color{red}with the results} in Test 1, the irrational spectral method on triangles shares %{\color{red}exhibit}
%the same asymptotic convergence as the polynomial spectral method on rectangles% for smooth problems
%.  Instead of an exponential convergence, a convergence rate of order roughly $3$ is observed in Fig.~\ref{ex2figs}, owing to the limited regularity of the exact solution.

{\small%\color{red}
\begin{table}[!ht]
		\caption{\small Comparison between the  approach in \cite{LWLM} and the new method}
		\vspace*{-7pt}
	\begin{center}
		\begin{tabular}{c*{4}{|c}}
			\hline
			\multirow{2}{*}{$N$} & \multicolumn{2}{c|}{Approach in \cite{LWLM}} & \multicolumn{2}{c}{Approach in this paper}\\
			\cline{2-5}
			& $\Lebsp{2}$-error & $\Lebsp{\infty}$-error & $\Lebsp{2}$-error & $\Lebsp{\infty}$ error\\% & LHS cond \#\\
			\hline
			15 & $2.866E\mathrm{-}06$ & $1.018E\mathrm{-}05$ & $2.349E\mathrm{-}06$ & $8.281E\mathrm{-}06$\\% & $1.250E\mathrm{+}04$\\%2.348942115649967e-006    8.280755233649289e-006    1.250470808050297e+004    4.828625359665802e-006
			\hline
			30 & $3.410E\mathrm{-}07$ & $1.203E\mathrm{-}06$ & $3.087E\mathrm{-}07$ & $1.091E\mathrm{-}06$\\% & $1.097E\mathrm{+}05$\\%3.086620921307916e-007    1.091052833585773e-006    1.096535496932601e+005    6.375531333154087e-007
			\hline
			45 & $9.940E\mathrm{-}08$ & $3.513E\mathrm{-}07$ & $9.299E\mathrm{-}08$ & $3.283E\mathrm{-}07$\\% & $3.979E\mathrm{+}05$\\%9.299484453340041e-008    3.283462084947973e-007    3.979081344186323e+005    1.922385457250491e-007 \\ 56: 4.857137510415174e-008    1.716431184432264e-007    7.978408825467364e+005    1.004301301543463e-007
			\hline
		\end{tabular}
			\label{ex22table}%moved to properly label the reference
	\end{center}
\end{table}
}

Finally, we compare our new approach with the method in \cite{LWLM} (where
the explicit consistency  condition  \eqref{ptsingular}  was built in the approximation space).
 One can see from Table \ref{ex22table} that both approaches enjoy a similar convergence behavior.  We reiterate that the new method does not require to modify the basis function, so with a pre-computation of the stiffness matrix, the triangular element can be treated as efficiently as the quadrilateral element.

\def \mct {\mathbf T}
\def \mcq {\mathbf Q}

\subsection{Concluding remarks} We initiated in this paper a new TSEM through presenting  the detailed implementation
and analysis on a triangle.  We demonstrated that the use of the rectangle-triangle mapping in
\cite{LWLM} led to much favorable grid distributions, when compared with the commonly-used Duffy's transform. More importantly, we showed the induced  singularity could be fully removed.  It is anticipated that with this initiative, we can develop an efficient TSEM on unstructured meshes built on a suitable discontinuous Galerkin formulation. This  will be discussed in a forthcoming work.

\backmatter

\bibliography{spectraltriangle}%,reftri}%{} %causes hyperref problem?
\bibliographystyle{plain}
\end{document}